\documentclass[oneside]{amsart}
\usepackage {CJKutf8}
\usepackage{amssymb}
\usepackage{amsthm}
\usepackage[abbrev]{amsrefs}
\usepackage{hyperref}

\usepackage{enumitem}

\usepackage{todonotes}

\usepackage{float} 

\usepackage{geometry}
\theoremstyle{plain}%
\newtheorem{theorem}{Theorem}
\newtheorem{proposition}[theorem]{Proposition}%
\newtheorem{lemma}[theorem]{Lemma}
\newtheorem{corollary}{Corollary}

\theoremstyle{remark}%
\newtheorem{remark}{Remark}%

\theoremstyle{definition}%
\numberwithin{equation}{section}
\numberwithin{theorem}{section}
\numberwithin{corollary}{section}

\newcommand{\inn}[2]{\left\langle#1,\,#2\right\rangle}

\DeclareMathOperator{\supp}{supp}

\newcommand{\grad}{\nabla}
\newcommand{\Lap}{\Delta}
\newcommand{\di}{\partial}

\DeclareMathOperator{\diam}{diam}

\newcommand{\br}[1]{\left\langle#1\right\rangle}
\newcommand{\si}{\sigma}
\newcommand{\eps}{\epsilon}
\newcommand{\ls}{\lesssim}

\newcommand{\al}{\alpha}

\newcommand{\Cb}{\mathbb{C}}

\newcommand{\Rb}{\mathbb{R}}

\newcommand{\one}{\ensuremath{\mathbf{1}}}

\newcommand{\hf}{\ensuremath{\mathfrak{h}}}

\renewcommand{\l}{\lambda} 

\newcommand{\abs}[1]{\ensuremath{\left\lvert#1\right\rvert}}
\newcommand{\norm}[1]{\ensuremath{\left\lVert#1\right\rVert}}
\newcommand{\sbr}[1]{\left[#1\right]}
\newcommand{\Set}[1]{\left\{#1\right\}}
\newcommand{\fullfunction}[5]{\ensuremath{
		\begin{array}{ccrcl}
			{#1}    & \colon  & {#2} & \longrightarrow & {#3} \\
			\mbox{} & \mbox{} & {#4} & \longmapsto     & {#5}
\end{array}}}

\newcommand{\md}[6]{\ensuremath{
		\ifinner
		\tfrac{\partial{^{#2}}#1}{\partial{#3^{#4}}\partial{#5^{#6}}}
		\else
		\tfrac{\partial{^{#2}}#1}{\partial{#3^{#4}}\partial{#5^{#6}}}
		\fi
}}
\newcommand{\del}[1]{\left(#1\right)}
\newcommand{\thmref}[1]{Theorem~\ref{#1}}

\newcommand{\secref}[1]{Section~\ref{#1}}
\newcommand{\lemref}[1]{Lemma~\ref{#1}}
\newcommand{\propref}[1]{Proposition~\ref{#1}}
\newcommand{\remref}[1]{Remark~\ref{#1}}
\newcommand{\figref}[1]{Figure~\ref{#1}}
\newcommand{\corref}[1]{Corollary~\ref{#1}}

\usepackage[normalem]{ulem}
\usepackage{soul}
\setstcolor{blue}

\definecolor{green}{rgb}{0.0, 0.5, 0.5}

\definecolor{lgray}{gray}{0.9}
\definecolor{llgray}{gray}{0.95}
\definecolor{lllgray}{gray}{0.975}

\newcommand{\R}{\mathbb{R}}

\newcommand{\cA}{\mathcal{A}}
\newcommand{\cB}{\mathcal{B}}
\newcommand{\cD}{\mathcal{D}}

\newcommand{\cE}{\mathcal{E}}
\newcommand{\cF}{\mathcal{F}}

\newcommand{\cN}{\mathcal{N}} 

\newcommand{\cX}{\mathcal{X}}

\newcommand{\nc}{\newcommand}

\nc{\h}{\delta}
\nc{\G}{\Gamma}
\nc{\et}{\eta} 
\nc{\gam}{\gamma}
\nc{\ka}{\kappa}
\nc{\lam}{\lambda}
\nc{\Lam}{\Lambda}

\nc{\ta}{\tau}
\nc{\w}{\omega}
\nc{\io}{\iota}
\nc{\s}{\sigma}
\nc{\vphi}{\varphi}

\nc{\e}{\epsilon}

\renewcommand{\k}{\kappa}

\nc{\ran}{\rangle}
\nc{\lan}{\langle}

\newcommand{\Ran}{\operatorname{Ran}}

\renewcommand{\Im}{\mathrm{Im}} 

\nc{\bfone}{{\bf 1}}
\nc{\1}{{\bf 1}}

\nc{\dd}{\mathrm{d}}

\newcommand{\DETAILS}[1]{}

\newcommand{\x}{\lan x\ran}

\DeclareMathOperator{\Ad}{ad}
\newcommand{\ad}[3]{\Ad^{#1}_{#2}(#3)}

\newcommand{\op}{\mathrm{op}}
\newcommand{\cp}{\mathrm{c}}

\newcommand{\Rem}{\mathrm{Rem}}
\DeclareMathOperator{\dG}{\mathrm{d}\Gamma}
\DeclareMathOperator{\Lip}{\mathrm{Lip}}

\newcommand{\Norm}[1]{{\left\vert\kern-0.25ex\left\vert\kern-0.25ex\left\vert #1 
		\right\vert\kern-0.25ex\right\vert\kern-0.25ex\right\vert}}

\newcommand{\wndel}[1]{\inn{\psi_0}{{#1}\psi_0}}

\newcommand{\wtdel}[1]{\inn{\psi_t}{{#1}\psi_t}}

\newcommand{\A}{\mathcal{A}}

\title{
	Spectral localization estimates for abstract linear Schrödinger equations
}

\begin{document}
 
\title[Spectral localization estimates]{Spectral localization estimates for abstract linear Schrödinger equations
}

\author{Jingxuan Zhang}
\address{Yau Mathematical Sciences Center\\
	Tsinghua University\\
	Haidian District\\
	Beijing 100084, China }
\address{Department of Mathematical Sciences\\
	University of Copenhagen\\
	Universitetsparken 5\\
	2100 Copenhagen, Denmark}
\email{jingxuan@tsinghua.edu.cn}

\date{\today}
\subjclass[2020]{35Q41   (primary); 35B40  , 35B45 ,81U90 ,  37K06      (secondary)}
\keywords{Schr\"odinger equations; A priori estimates}

\pagestyle{plain}
\maketitle
\begin{abstract}
	We study the propagation properties of abstract linear Schr\"odinger equations of the form $i\di_t\psi = H_0\psi+V(t)\psi$, where $H_0$ is a self-adjoint operator and $V(t)$ a time-dependent potential. We present explicit sufficient conditions ensuring that if the initial state $\psi_0$ has spectral support in $(-\infty,0]$ with respect to a reference self-adjoint operator $\phi$, then, for some $c>0$ independent of $\psi_0$ and all $t\ne0$, the solution $\psi_t$ remains spectrally supported in $(-\infty,c\abs{t}]$ with respect to $\phi$, up to an $O(\abs{t}^{-n})$ remainder in norm. The main condition is that the multiple commutators of $H_0$ and $\phi$ are uniformly bounded in operator norm up to the $(n+1)$-th order. We then apply the abstract theory to a class of nonlocal Schr\"odinger equations on $\Rb^d$, proving that any solution with compactly supported initial state remains approximately supported, up to a polynomially suppressed tail in $L^2$-norm, inside a linearly spreading region around the initial support for all $t\ne0$. 
	
\end{abstract}

	\section{Introduction}

We consider the following  abstract linear Schr\"odinger equation on a Hilbert space $\hf$:
\begin{align}\label{SE}
	i\di_t\psi = H(t)\psi,\quad H(t)=H_0+V(t).
\end{align}
 Here $H_0$  is a densely defined,  self-adjoint operator  on $\hf$ and $V(t)$ is a time-dependent potential such that $H(t)$ admits bounded propagator on $\hf$ for all times. 

Our aim in this paper is to control the spectral localization properties of states evolving according to \eqref{SE}. Specifically, fix a reference self-adjoint operator $\phi$ on $\hf$ and let  $P_{a}$ be the spectral projection of $\phi$ onto the half-line $(a,\infty)$. 
Under suitable conditions on the commutator between $\phi$ and the system Hamiltonian, we prove that any solution $\psi_t,\,t\in\Rb$ to \eqref{SE} with initial state $\psi_0$ satisfies, for some $c,\,C>0$  independent of  $\psi_0$ and all $t\ne0$,
\begin{align}\label{locEst}
	\norm{P_{c\abs{t}} \psi_t}_\hf\le C \del{\norm{P_{0}\psi_0}_\hf+\abs{t}^{-n}\norm{\psi_0}_{\hf}}.
\end{align}
Estimate \eqref{locEst} ensures that if the initial state $\psi_0$ has spectral support in $(-\infty,0]$ w.r.t.~the reference operator $\phi$, then the corresponding solution $\psi_t$ remains spectrally localized in $(-\infty,c\abs t]$ w.r.t.~$\phi$ up to a polynomially decaying remainder in norm. See Theorem \ref{thmMain} for precise statement and the subsequent remarks for generalizations. The physical significance is discussed in \secref{secLit}.


We then apply the general theory \eqref{locEst} to study spacetime localization properties for some concrete models of nonlocal dispersive equations. We work in $\hf=L^2(\Rb^d)$ and, for $X\subset\Rb^d$, $a\ge0$,  denote by $\1_X$ the characteristic function of $X$, 
$$X_{a}:=\Set{x\in\Rb^d:d_X(x)\le a},\quad\text{and}\quad X_{a}^\cp:= \Rb^d\setminus X_{a}.$$
We choose the reference operator $\phi=d_X$, the multiplication operator by the distance functions $d_X(x):=\inf_{y\in X}\abs{y-x}$, so that $P_{a}$ amounts to the multiplication operator by $\one_{X_a^\cp}$.
The Hamiltonians $H(t)=H_0+V(t)$ are given by a class of nonlocal non-autonomous operators on $L^2(\Rb^d)$, with $H_0$ satisfying suitable finite moment bounds and potential $V(t)$ given by arbitrary uniformly bounded multiplication operators. 
Under condition \eqref{k-cond0} on the kernel of $H_0$, we prove that if the wave function $\psi_t$ evolving according to such a nonlocal Schr\"odinger equation is supported in $X$ at $t=0$, then there holds the dispersive estimate
\begin{align}
	\label{locEstL2}
	\sup_{t\ne 0} \abs{t}^n	\int_{X_{c\abs t}^\cp}\abs{\psi_t}^2\le C\norm{\psi_0}_{L^2}.
\end{align}
From here we conclude the Strichartz-type estimates 
\begin{align}
	\label{StrEst}
	\norm{\one_{X_{c\abs{t}}^\cp}\psi_t}_{L^p_t(L^2_x)}\le C \norm{\psi_0}_{L^2},
\end{align} for all $p>1/n$ and $\psi_0$ with $\supp \psi_0\subset X$. 
Furthermore, {by Markov's inequality and \eqref{locEstL2}, we have 
	\begin{align}
		\label{tailEst}
		\sup_{t\ne0}\mu\bigl(\bigl\{{{x\in X_{c\abs{t}}^\cp}:\abs{\psi_t}^2\ge \abs{t}^{-n}}\}\bigr)
		\le C\norm{\psi_0}_{L^2}.
	\end{align}
Localization estimates \eqref{locEstL2}--\eqref{tailEst} impose direct constraints on the size of the probability tails.}
See Section \ref{secApplRes} for the precise results.


{The novelty of our results in this paper is twofold. Firstly, we identify explicit sufficient conditions ensuring spectral localization estimate \eqref{locEst} for abstract Schr\"odinger equations (see conditions \eqref{commVphi}--\eqref{commHphi} below). In particular, the system Hamiltonian $H(t)$ does not enter the conditions directly, but only through its commutators with the reference operator $\phi$ appearing in \eqref{locEst}. Secondly, comparing to recent results \cite{MR4254070,MR4604685} where similar propagation estimates as \eqref{locEst} are obtained for the standard Schr\"odinger equation and the Hartree equation, our results hold in an abstract  Hilbert space setting and complement existent results in the case where the Hamiltonian involves nonlocal operators and the reference operator $\phi$ is not given by specific forms. }

\subsection{Organization of the paper}
In Section \ref{secRes}, we present our main results, illustrate the key technical steps, and discuss the methodology in the context of relevant literature.   
In Section \ref{chapProofs}, we furnish the proofs of the main results, Theorems \ref{THMRME}--\ref{THMAP}, in the general setting.

In Section \ref{chapMVENonLoc}, we illustrate applications to a large class of nonlocal Schr\"odinger equations on $\Rb^d$. The models under consideration have the favourable property that they satisfy the main condition laid out in Section \ref{secRes}  with any multiplication operator $\phi$ by functions in the homogeneous Sobolev space $\dot W^{1,\infty}$ (i.e., weakly differentiable with $L^\infty$-gradient). These could be viewed as typical models for nonlocal Hamiltonians, as the fractional Laplacian $H=(-\Lap)^{1/2}$ does enjoy the same property, owning to the theory of Calder\'on commutators. 

In Section \ref{chapPre}, we complete the proof of some technical estimates needed to establish certain expansion formulae in Section \ref{chapProofs}. In the appendix, we prove certain commutator bounds for the nonlocal Hamiltonians studied in Section \ref{chapMVENonLoc}.

\subsection*{Notation}\label{secNota}
	We denote by $\1$ the identity operator, $\cD(A)$ the domain of an operator $A$, and  
$\|\cdot\|$   the norm of  operators on $\hf$ and sometimes that of vectors in $\hf$.
$\cB(\hf)$ denotes the space of bounded operators on $\hf$.
We make no distinction in our notation between a function $f\in\hf$ and the associated multiplication operator $\psi(x)\mapsto f(x)\psi(x)$ on $\hf$.

The commutator $[A,B]$ of two operators $A$ and $B$ is first defined as a quadratic form on $\cD(A)\cap\cD(B)$ (always assumed to be dense in $\hf$) and then extended to an operator. Similarly, the multiple commutators of $A$ and $B$ are defined recursively by $\ad{0}{B}{A}=A$ and $\ad{p}{B}{A}=[\ad{p-1}{B}{A},B]$ for $p=1,2,\ldots$.

\section{Setup and main result}\label{secRes}

Let $\hf$ be a complex Hilbert space equipped with inner product $\inn{\cdot}{\cdot}$ and induced norm $\norm{\cdot}$.
We consider a dynamical system described by \eqref{SE}, with $H(t)$ defined on a common dense domain $\cD=\cD(H_0)$ for all times. Furthermore we assume \eqref{SE} is globally well-posed on $\hf$. By standard perturbation theory, a simple sufficient condition is that $H_0$ is a densely defined self-adjoint operator on $\hf$ and $V(t)$ is uniformly bounded for all $t$. 

We will mainly work in the Heisenberg picture and study the Heisenberg evolution, $\al_t$, for differentiable families $A(t)\in\cB(\hf)$, characterized by the duality relation
\begin{align}
	\label{Aevol}
\inn{\psi_0}{\al_t(A(t))\psi_0}=\inn{\psi_t}{A(t)\psi_t} ,
\end{align}
where $\psi_t,\,t\in\Rb$ is the unique global solution to the Schr\"odinger equation \eqref{SE} with $\psi_t\vert_{t=0}=\psi_0\in\cD$. 

\subsection{Main result}\label{secMainRes}
Throughout the paper, we fix a reference self-adjoint operator $\phi$ on $\hf$, such that $\cD(H_0)\cap\cD(\phi)$ is dense in $\hf$.  Our main assumptions are then stated in terms of commutators between the reference operator $\phi$ and the system Hamiltonian in \eqref{SE}. 

First, for some integer  $n\ge1$, we assume that $\phi$ `almost commutes' with the potential, in the sense that $[\phi,V(t)]$ extends to bounded operators for all times and satisfies, for some fixed $C_V>0$, 
\begin{align}
	\label{commVphi}
\norm{G}_{L^1(\Rb)}\le C_V\quad \text{ where }\quad G(t):= \norm{[\phi,V(t)]} .\tag{C1}
\end{align}
Second, we assume that the multiple commutators 
$\ad{p}{\phi}{H_0},\,p=1,\ldots,n+1$, all extend to bounded operators on $\hf$; namely, there exist $\kappa_1,\ldots,\kappa_{n+1}>0$ such that 
\begin{align}\label{commHphi}
	\norm{\ad{p}{\phi}{H_0}} \le \kappa_p\qquad (p=1,\ldots,n+1).\tag{C2}
\end{align}

To state our main result, let \begin{align}\label{kappaDef}
	\kappa\equiv \kappa_1,
\end{align}
so that, by \eqref{commVphi}, $\kappa$ amounts to the norm of the generalized group velocity operator $i[H,\phi]$. 
Recall that $P_{a},\,a\in\Rb$ is the spectral projection of $\phi$ onto  $(a,\infty)$, explicitly, 
\begin{align}\label{Pdef}
	P_{a}\equiv\1_{(a,\infty)}(\phi).
\end{align} 
The main result of this paper is the following:
\begin{theorem}[Spectral localization]  \label{thmMain}
	Suppose \eqref{commVphi}--\eqref{commHphi} hold for some $n\ge1$. 
	Then, for any $ c > \kappa$, there exists $C=C(n,c)>0$ such that for all $t\ne0$, the following operator inequality holds on $\hf$:
		\begin{equation}
			\label{eLCgen}
		\boxed{	\al_t\del	{P_{{c\abs{t}}}}\le P_0 + C(P_0+C_V)\abs{t}^{-1} + C\abs{t}^{-n}.} 
	\end{equation}
\end{theorem}
This theorem is proved at the end of \secref{secIdea}

\begin{remark}
	Let $C_V=0$ in \eqref{commVphi}. Then	by the duality relation \eqref{Aevol}, it is readily verifiable that \eqref{eLCgen} is equivalent to the localization estimate \eqref{locEst} in the Schr\"odinger picture. 
\end{remark}

\begin{remark}
	For $[H_0,\phi]=0$, we have $\kappa=0$ and for large $t$, the propagation estimate \eqref{eLCgen} affirms the fact that $\phi$ is conserved along the evolution \eqref{SE}. 

\end{remark}

\begin{remark}
	Replacing $\phi\to -\phi$ in \eqref{Pdef} and setting $P_a^-:=\1_{(-\infty,a)}$, we conclude from \eqref{eLCgen} that
				\begin{equation}\label{eLCgen-}
		{	\al_t\del	{P_{{-c\abs{t}}}^{-}}
			\le  P_0^- + C(P_0+C_V)\abs{t}^{-1} + C\abs{t}^{-n}}.
	\end{equation}
	Similarly, shifting the reference operator $\phi\to \phi-b$ yields
					\begin{equation}\label{eLCgenb}
		{	\al_t\del	{P_{{b+c\abs{t}}}}
			\le P_b + C(P_b+C_V)\abs{t}^{-1} + C\abs{t}^{-n} },\quad b\in\Rb.
	\end{equation}
	Note that both ways of modification above have no bearing on conditions \eqref{commVphi}--\eqref{commHphi}.
\end{remark}

\begin{remark}
	In	\cite{MR1720738}, Hunziker, Sigal, and Soffer proved a corresponding `minimal escape velocity' bound which complements estimate \eqref{eLCgen}: Suppose $C_V=0$ in \eqref{commVphi} and condition \eqref{commHphi} holds for some $n\ge1$. Assume further $i[H_0,\phi]\ge \theta $ for some $\theta >0$. Then, by \cite[eq.~(1.11)]{MR1720738}, for any $0<\vartheta<\theta$ and $a\in\Rb$, there exists $C>0$ s.th.~for all $t>0$,
	\begin{equation}\label{eLCgenLow}
		\al_t\del	{P_{{a+\vartheta{t}}}^-}
		\le    {P_{{ a}}^-+C{t}^{-n} }.
	\end{equation}
	Combining \eqref{eLCgenLow} and \eqref{eLCgenb} with $C_V=0$, we obtain a refined two-sided control on the propagation of spectral supports of evolving states: Fix $a>b$ and assume the initial state $\psi_0\in \Ran \1_{(a,b)}(\phi)$. Then, for any $\vartheta<\theta\le \kappa <c$, the solution $\psi_t,\,t>0$ generated by $\psi_0$ remains localized in the linearly expanding `annular regions' $\Ran \1_{(a+\vartheta t,b+ct)}(\phi)$, up to an $O(t^{-n})$ remainder in norm.
\end{remark}

\subsection{Literature review and discussion}\label{secLit}
The question of \textit{locality} versus \textit{nonlocality} has always been a fundamental issue in the study of physical phenomena and the corresponding mathematical models.  On $\hf=L^2(\Rb^d)$, 
our results yield, through \eqref{locEstL2}--\eqref{tailEst},  spacetime localization estimates for the propagation of information in a range of non-relativistic quantum mechanical models. Establishing such estimates is a delicate matter, because the dispersive structure of the governing evolution equations generally leads to an apparent lack of locality in the restrictive sense.  For instance, suppose one defines, as usual, the maximal propagation speed as the infimum of all $c$'s such that states initially supported in $X\subset \Rb^d$ at time $t=0$ remains supported, for all times $t>0$, in the light cone $X_{ct}$. Then, even for the simplest example of a free particle evolving according to the Schr\"odinger equation $i\di_t\psi=-\Lap \psi$, infinite speed of propagation can be observed by examining the Fourier transform of the solution and using the superlinear growth of the dispersion relation. 
This general idea also leads to infinite speed of propagation, with the usual definition above, for typical $1$-body quantum evolutions described by a large class of dispersive equations  \cite{MR1264524,MR1443322}.

Historically, for quantum evolutions described by the standard Schr\"odinger equations with Hamiltonians $H=-\Lap+V$, where the potential $V$ is sufficiently regular, one approach to recover an appropriate sense of locality is to introduce an energy cutoff adapted to the spectrum of $H$ on the initial state, and then show that the probability of finding the (microlocalized) state in the classically forbidden region vanishes asymptotically in time. Thus, localization properties of evolving states are reformulated in terms of propagation estimates for certain time-dependent observables identifying the spacetime support of states at time $t$, and finite speed of propagation is established in terms of the resulting propagation estimates. 
This idea is also the starting point of the present paper.

More precisely, V.~Enss proved in his seminal works  \cite{En1,En2} that if a particle with unit mass is initially localized in a ball $X$ at $t=0$ and has energy below $c^2/2$, then  the probability, $p(t)$, of finding the particle at time $t>0$ in the classically forbidden region $X_{ct}^\cp$ vanishes as a $L^1$ function, i.e., $\int p(t)\,dt<\infty$.
This way one obtains {effective light cones} (ELCs), viz., regions outside of which the probability of finding the particle  vanishes asymptotically in time. 
Notice that the ELCs obtained this way are energy-dependent, in agreement with the physical intuition that a particle should move at a speed proportional to the square root of its energy. 

The result of Enss was subsequently improved in \cite{SigSof,Skib} and, more recently, \cite{MR4254070,MR4604685}, to Schr\"odinger equations with time-dependent Hamiltonians and the Hartree equation. Historically, such propagation properties have played crucial roles in scattering theory of Schr\"odinger operators, leading to important breakthroughs in the study of asymptotic completeness of $N$-body problems by Enss \cite{En3, En4}, Skibsted \cite{Skib2, Skib3}, and Sigal-Soffer \cite{SigSof1,SigSof2,SigSof3,SigSof4}, among many others. For reviews of the development in scattering theory along this line, see \cite{HunSig2,HunSig3}.

Our approach to derive the main propagation estimate \eqref{locEst} is a generalization of the monotonicity method originated from the classical works of Sigal and Soffer's \cite{SigSof,SigSof1,SigSof2,SigSof3,SigSof4} in scattering theory, in which the authors proved that for general  time-dependent Schr\"odinger equations on $\Rb^d$, evolving states admit ELCs that spread out in space  at a finite rate.
The results of the seminal works \cite{SigSof,SigSof2} were improved in \cite{Skib, MR1335378, MR4254070, MR4604685} and  extended to non-relativistic QED  
in \cite{BoFauSig}, open quantum systems in \cite{MR4529865,MVBvNL}, nonlinear equations in \cite{MR4604685} and condensed matter physics in \cite{FLS1, FLS2,LRSZ}.  Our result is motivated and built upon the works mentioned above and more.

\subsection{Key steps for proving \thmref{thmMain}}\label{secIdea}
In this section, we  illustrate the key steps in the proof of our main result, the operator inequality \eqref{eLCgen}.

Our approach to study the localization properties of quantum evolutions, 
pioneered in \cite{SigSof, Skib, MR1335378, MR4254070, BoFauSig,FLS1, FLS2}, is  based on approximate monotonicity formulae for certain propagation-identifying observables. As the starting point,  we define, for any differentiable family $A(t)\in\cB(\hf)$, the Heisenberg derivative 
\begin{align}
	\label{HeisD}&D_{(\cdot)} A(t):=\di_tA(t)+i[(\cdot), A(t)].
\end{align}
It is readily verified that the evolution $\al_t$, defined by relation \eqref{Aevol}, solves the Heisenberg equation
\begin{align}\label{HE}
	\di_t(\al_t(A(t))) = \al_t(D_H(A(t))).
\end{align} 
By the fundamental theorem of calculus and the fact that $\al_0(A(0))=A(0)$, it follows that 
\begin{align}\label{intHE}
	\al_t(A(t))=A(0)+\int_0^t\al_t(D_H(A(r)))\,dr.
\end{align}
Our goal is to establish monotonicity estimates, based on \eqref{intHE}, for the evolution  of suitable time-dependent observables $A(t)$. Roughly speaking, we design $A(t)$ so that it satisfies the following properties:
\begin{itemize}
	\item$\al_t(A(t))$ is constrained by a time-decaying envelope, with $\al(A(t))\ls A(0)+\abs{t}^{-n}$.  
	\item $ A(t) $ is comparable with $P_{c\abs{t}}$, with  $P_{c\abs{t}}\le A(t)$ for $t\ne0$ while $A(0)\le P_0$. 
\end{itemize}
Combining these, we arrive at the desired estimate \eqref{eLCgen}. 

  Below we elaborate on the key steps in establishing these properties.

\subsubsection{Construction of ASTLOs.}

We fix a test light-cone slope $c>0$ and  let $s>0$ be a large adiabatic parameter.  For the reference operator $\phi$ entering \eqref{commVphi}--\eqref{commHphi},  times $t\in\Rb$, and   smooth cutoff functions $\chi$ in a suitable class $\cX\subset C^\infty\cap L^\infty(\Rb,\Rb_{\ge0})$ (see \figref{chifig1} below), we define the \textit{adiabatic spectral localization observables}, or \textit{ASTLOs} (adopting the terminology of \cite{FLS1,FLS2}), as
\begin{equation}\label{Astchi}
	\A_s(t,\chi):=\chi\del{\frac{\phi-c\abs{t}}{s}}.\tag{ASTLO}
\end{equation} 	
The precise construction of the class $\cX$ is rather flexible and not so relevant to the present exposition; we defer it to \eqref{F} where the detailed proof is presented.
Note that $\A_s(\cdot)$ is an operator-valued function defined by functional calculus of the self-adjoint operator $s^{-1}(\phi-c\abs{t})$. 

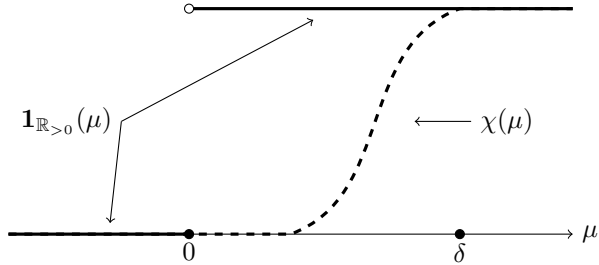
\begin{figure}[H]
	\centering
	\begin{tikzpicture}[scale=3]
		\draw [->] (-.5,0)--(2,0);
		\node [right] at (2,0) {$\mu$};
		\node [below] at (.3,0) {$0$};
		\draw [fill] (.3,0) circle [radius=0.02];
		
		
		\node [below] at (1.5,0) {$\delta $};
		\draw [fill] (1.5,0) circle [radius=0.02];

		\draw [very thick] (-.5,0)--(.3,0);
		\draw [very thick] (.3,1)--(2,1);				
		\filldraw [fill=white] (.3,1) circle [radius=0.02];
		
		\draw [dashed, very thick] (-.5,0)--(.75,0) [out=20, in=-160] to (1.5,1)--(2,1);

		\draw [->] (1.55,.5)--(1.3,.5);
		\node [right] at (1.55,.5) {$\chi(\mu)$};
		
		\draw [->] (0,.5)--(.85,.95);
		\draw [->] (0,.5)--(-.05,.05);
		\node [left] at (0,.5) {$\one_{\Rb_{>0}}(\mu)$};
	\end{tikzpicture}
	\caption{A typical function $\chi\in \cX$ compared with the characteristic function of $\Rb_{>0}$. Here $\delta>0$ is a parameter entering the definition of $\cX$ through \eqref{RMEproto} below. In essence, $\chi$ is a smoothed-out version of $\one_{\Rb>0}$ with derivative supported in $(0,\delta)$.}\label{chifig1}
\end{figure}

\begin{remark}
	Regarding the nomenclature of `ASTLO', we say that $\cA_s(t,\chi)$ is adiabatic since, for a test light-cone slope $c=O(1)$ and a large adiabatic parameter  $s\gg1$, the velocity $\di_t\A_s(t,\chi)=-cs^{-1}\A_s(t,\chi')=O(s^{-1})$ varies at a slow scale. 
To see that $\cA_s(t,\chi)$ identifies the spectral localization property of states, one can view $\cA_s(\cdot,\chi)$ as a smoothed spectral cutoff function associated to $\phi$ and various half-lines (see \figref{chifig1}).
Indeed, for functions $\chi$ in appropriate classes and suitably chosen $s$, we can arrange to have
\begin{align}\label{chi-0s-est'} 
	\A_s(0,\chi)  \le&   P_{0},\\
	\label{chi-ts-est'}P_{c\abs{t}} \le &\A_s(t,\chi)\quad (t\ne0).
\end{align}
The detailed relations  are formulated and proved in \propref{propGeo}.
\end{remark}

\subsubsection{Differential inequality for ASTLOs.}

By relation \eqref{HE} and the almost-commuting assumption \eqref{commVphi}, 
it is easy to verify that for some absolute constant $C>0$,
\begin{align}
	\label{HeisDrel}
	\di_t\al_t(\A_s(t,\chi))\le \al_t(D_{H_0}\A_s(t,\chi))+Cs^{-1}G(t).\tag{H}
\end{align}
See \lemref{lem5.5} and \eqref{217} for details.
Further exploiting relation \eqref{HeisDrel}, we obtain
\begin{theorem}[Recursive monotonicity estimate]\label{THMRME}Suppose the evolution $\al_t$ satisfies identity \eqref{HeisDrel}, and condition \eqref{commHphi} holds for some $n\ge1$.  
	Then, for any $ c > \kappa$ and  $\chi\in\cX$, there exists $C>0,\,\xi\in\cX$ depending only on $n,\chi$ such that for $\delta:=c-\kappa>0$ and all $s>0,\,t\in\Rb $:
	\begin{align}\label{RMEproto} \di_t\al_t\del{\A_s(t,\chi)} 
		\leq& -\delta s^{-1}{\al_t\del{\A_s(t,\chi')}}+ Cs^{-2}{\al_t\del{\A_s\del{t,\xi'}}}+ Cs^{-1}\del{s^{-n}+ G(t)}.
		\tag{RME}
	\end{align}
\end{theorem}
This theorem is proved in \secref{sec:pfRME}.
\begin{remark}
	Note that for the free evolution with $V(t)\equiv 0$, relation \eqref{HeisDrel} follows with equality immediately from definitions \eqref{Aevol} and \eqref{HeisD}. See discussions in \secref{secExt} and concrete examples in \secref{chapMVENonLoc} in which $V(t)$ is nontrivial.
\end{remark}

\begin{remark}
	The differential inequality \eqref{RMEproto} is `recursive monotone' because the second  term on the r.h.s.~is of the same form as the leading  term, and the latter is non-positive. Notice that \eqref{RMEproto} (more precisely, the proof of it) is the only place where information of the evolution  $\al_t$ is used. The only property we require of the underlying evolution is the differential identity \eqref{HeisDrel}, which  asserts that the velocity $\di_t \al_t(\A_s)$ is approximately determined by the Heisenberg derivative  $D_{H_0}\A_s$ corresponding to the autonomous part $H_0$ in the system Hamiltonian, up to a small time-decaying remainder due to the potential. 
\end{remark}

To prove \thmref{THMRME}, we use the commutator bounds \eqref{commHphi} to derive an expansion formula for $i[H_0,\A_s]$ in terms of the bounded multiple commutators entering \eqref{commHphi}.  Combining this expansion with the explicit form of $\di_t\A_s$ and \eqref{HeisDrel} then yields \eqref{RMEproto}.

\subsubsection{Monotonicity estimate for ASTLOs.}
Since the evolution $\al_t$ is positivity-preserving, control over $\al_t(A_s(t,\chi))$ along $t$ yields the same over $\al_t(P_{c\abs{t}})$ through relations \eqref{chi-0s-est'}--\eqref{chi-ts-est'}. Thus, our goal now is to derive monotonicity estimates for $\al_t(\A_s(t,\chi))$. Indeed, through \eqref{RMEproto}, we can control the growth of $\A_s(t,\chi)$ as follows:
\begin{theorem}[monotonicity estimate]\label{THMME}
	Suppose \eqref{RMEproto} holds for $n\ge1$ and \eqref{commVphi} holds for some $C_V>0$.  
	Then, for any $ c > \kappa$ and  $\chi\in\cX$, there exists $C>0,\,\xi\in\cX$ depending only on $n,c,\chi$ such that for  all  $s>0$, $t\in\Rb$:
	\begin{align}\label{MonoEst}
		\al_t\del{\A_s(t,\chi)}  \le	{\A_s(0,\chi)}+ {Cs^{-1}} {\A_s(0,\xi)} +Cs^{-1}(s^{-n}\abs{t}+C_V) .\tag{ME}
	\end{align}
\end{theorem}
This theorem is proved in \secref{sec:pfMonoEst}.

\begin{remark}
	Estimate \eqref{MonoEst} shows that the expectation of $\A_s(t,\chi)$ is bounded by a time-decaying envelope for $\abs{t}\le s$; see Figure \ref{evelPic} below.  Indeed, assume for simplicity $C_V=0$. We evaluate the expectation of both side of \eqref{MonoEst} on a state $\psi\in\cD$ and use the duality $\wndel{\al_t(A)}=\wtdel{A}$ (see \eqref{Aevol}) to find, for $\abs{t}\le s$,
\begin{align}\label{MEstate}
	\wtdel{\A_s(t,\chi)}  \le	\wndel{\A_s(0,\chi)}+C{\abs{t}^{-1}}\wndel{\A_s(0,\xi)}+C\abs{t}^{-n}\norm{\psi_0}^2.
\end{align}

\end{remark}
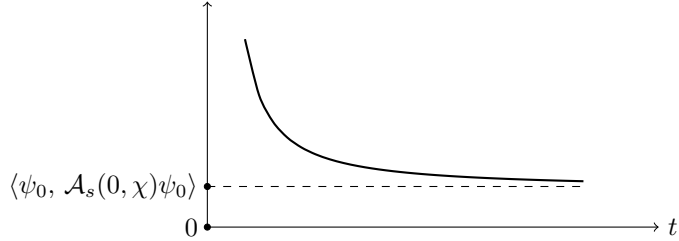
\begin{figure}[H]
	\centering
	\begin{tikzpicture}
		\draw[->] (0, 0) -- (6, 0) node[right] {$t$};
		\draw[->] (0, 0) -- (0, 3);
		\draw[scale=0.5, domain=1:10, smooth, variable=\x, thick] plot ({\x}, {1+ 2/\x+2/(\x*\x)});
		\draw[dashed] (0, .54) -- (5, .54) ;
		\node[left] at (0, .54) {$	\wndel{\A_s(0,\chi)} $};
		\draw [fill] (0, .54) circle [radius=0.04];
		\node[left] at (0, 0) {$0 $};
		\draw [fill] (0, 0) circle [radius=0.04];
	\end{tikzpicture}
	\caption{Schematic diagram illustrating the monotone envelop on the r.h.s. of~\eqref{MEstate}.}\label{evelPic}
\end{figure}


Here we sketch the proof of \eqref{MonoEst} for $t\ge0$. Integrating  \eqref{RMEproto}, dropping certain non-negative terms, and using the $L^1$ bound on $G(t)$ from \eqref{commVphi}, we find
\begin{align}
	\delta s^{-1}\int_0^t{\al_t\del{\A_s(t,\chi')}}\leq& \A_s(0,\chi)+C {s^{-2}{\int_0^t \al_t\del{\A_s\del{t,\xi'}}}+ Cs^{-1}(s^{-n}t+C_V)}, \label{RME1}\\
	\al_t\del{\A_s(t,\chi)}
	\leq&\A_s(0,\chi)+C  {s^{-2}\int_0^t {\al_t\del{\A_s\del{t,\xi'}}}+ Cs^{-1}(s^{-n}t+C_V)}.\label{RME2}
\end{align}
Since the l.h.s.~term in \eqref{RME1} is of the same form as the second term on the r.h.s., we can apply the same estimate to the latter, while introducing another cutoff function, say $\eta$,  in the same class $\cX$ as $\chi$ and $\xi$. Iterating this procedure $n$ times on \eqref{RME1} until no integral is present in the r.h.s., we obtain
\begin{align}\label{RME3}
	\int_0^t{\al_t\del{\A_s(t,\chi')}}	\le C\del{s\A_s(0,\chi)
		+ \A_s(0, \xi)+   \cdots +s^{-(n-2)}\A_s(0, \eta)+s^{-n}t+C_V},
\end{align}
for $n$ cutoff functions $\xi,\ldots, \eta\in\cX$. 
Lastly, we apply \eqref{RME3} to bound the integral in the r.h.s.~of \eqref{RME2} (which carries a prefactor of $s^{-2}$). This, together with some additional algebraic  properties of the ASTLOs,  yields \eqref{MonoEst}. See \secref{sec:pfMonoEst} for detailed derivations.

\subsubsection{Concluding estimate \eqref{eLCgen}.}

\thmref{THMME} establishes the approximate monotonicity for ASTLOs as alluded to in \secref{secIdea}. It remains now to conclude the desired propagation estimate \eqref{eLCgen}.

Using the geometric properties \eqref{chi-0s-est'}--\eqref{chi-ts-est'} 
and  estimate \eqref{MEstate} for $\A_s(t,\chi)$, we obtain
\begin{theorem} \label{THMAP}
	Suppose \eqref{MonoEst} holds for $n\ge1$. 
	Then, for any $ c > \kappa$, there exists $C=C(n,c)>0$ such that \eqref{eLCgen} holds for all $t\ne0$.
	
\end{theorem}
This theorem is proved in \secref{sec:pfMVE}. Note that the statement of \thmref{THMAP} no longer involves the ASTLOs.

\begin{proof}[Proof of \thmref{thmMain}]
Using the relation 
\begin{align}\label{217}
		\di_t\al_t(\A_s(t,\chi))=& \al_t(D_{H_0}\A_s(t,\chi))+\al_t(i[V(t),\A_s(t,\chi)])\notag\\
		\le &  \al_t(D_{H_0}\A_s(t,\chi)) + \norm {[V(t),\A_s(t,\chi) ]} ,
\end{align}
together with Lemma \ref{lem5.5}, which shows $\norm {[V(t),\A_s(t,\chi) ]}\le Cs^{-1}G(t)$ for some absolute constant $C>0$, 
we verify that \eqref{HeisDrel} holds.  
This, together with Thms.~\ref{THMRME}--\ref{THMAP}, yields inequality \eqref{eLCgen} under conditions \eqref{commVphi}--\eqref{commHphi} for $s\ge t$. Thus the proof is complete. 
\end{proof}

\subsection{Discussion}\label{secExt}


We are interested in the spectral localization properties of abstract dispersive evolutions. We have proposed a modular paradigm that accomplishes this goal, as long as the underlying evolution satisfies \eqref{HeisDrel} with a reference operator  satisfying  \eqref{commVphi}--\eqref{commHphi}. 
Our method is geometric in nature and traces back to the line of works by Enss, Hunziker, Sigal, Skibsted, Soffer, and others, who have laid out the foundation of the geometric method for scattering theory of the Schr\"odinger operators (see \cite{HunSig2,HunSig3} for reviews). 

Indeed, the asymptotic localization theory of general dispersive evolutions \eqref{SE}, which we consider here, bears an intrinsic similarity to the scattering theory of the standard Sch\"odinger operators. Both problems concern with the semiclassical behaviour of particles for large time. Notice however that the parameter $s>0$ in \eqref{Astchi}, essentially a semiclassical parameter, does not come with the model \eqref{SE}, but is determined by the problem directly. 
The precise choice of $s$ is given in \eqref{Adef}.

A main technical advantage of our localization theory based on the analysis of ASTLOs lies in its flexibility. Whereas strictly monotone quantities along a given evolution equation are rare to find, the approximately monotone ASTLOs are rather easy to engineer.  One reason is that the underlying evolution does not enter directly into our analysis, but only through assumptions \eqref{commVphi}--\eqref{commHphi} on the commutators between the reference operator $\phi$ and the system Hamiltonian. 

For example, suppose \eqref{SE} is posed on $\Rb^d$ and the  reference operator $\phi$ is a multiplication operator by a function $\phi:\Rb^d\to\Rb$ with $\supp \phi\subset X^\cp$ for some $X\subset \Rb^d$. Then, for generic pseudo-differential operators $\bar V(t)$ satisfying relevant domain conditions and $ V(t):= \1_X \bar V(t) \1_X$, we have $ V(t)\phi = \phi V(t)=0$, and so \eqref{commVphi} is satisfied with $C_V=0$. Consequently, the ASTLOs satisfy the differential identity \eqref{HeisDrel} with equality. Since the evolution only enters our analysis through \eqref{RMEproto} and the latter depends only on \eqref{commHphi} and  \eqref{HeisDrel}, we conclude that \eqref{RMEproto} and all subsequent modular theorems hold.  See Section \ref{chapMVENonLoc} for more details.

Moreover, since the system Hamiltonian  does not enter directly into the main technical assumption \eqref{commHphi}, but only through its commutators with $\phi$ in \eqref{Astchi}, we can derive conditional localization properties when the commutator assumption \eqref{commHphi} fails for the obvious choice of $\phi$.  Consider the case $H_0=-\Lap$ acting on $\Rb^d$. Let $d_X$ be a smoothed distance function to a smooth bounded domain $X$. The obvious choice $\bar \phi=d_X$ does not satisfies \eqref{commHphi}, since  $[-\Lap,d_X]=-\Lap d_X-2\grad d_X\cdot \grad $ is unbounded. However, with an energy cutoff $g=g_E(H_0)$, where $E\in\si(H_0)$ and $g_E$ is a smooth cutoff function supported in $\Rb_{\le E}$, one can check that, with the microlocalized position operator $\phi=g d_X g$, the (microlocal) group velocity  $i[H_0,\phi]$   (together with higher commutators) is bounded. Using this microlocalized version of $\phi$ in \eqref{Astchi} and running the paradigm above,  we obtain energy-dependent spacetime localization estimates for $H_0=-\Lap$.  {See \cite{MVBvNL} for concrete results of this nature, with applications to von-Neuman-Linblad equations in Markovian open quantum dynamics. Related propagation bounds involving microlocal cutoff are obtained in \cite{MR4254070,MR4604685} for linear and nonlinear quantum dynamics involving standard Schr\"odinger operators.}

Lastly, since our method is based on monotonicity estimate in the form of operator inequalities, we can reduce localization theory for quantum many-body problems to the corresponding $1$-body problems. Consider an abstract second quantization map, $\dG$, mapping $1$-body observables $A$ acting on $\hf$ to many-body observables $\hat A$ acting on a Fock space $\cF$ over $\hf$. We assume the map $\dG$ is positive-preserving, i.e., for any self-adjoint $1$-body operators $A$, $B$, 
\begin{align}
	A\le B\implies \hat A\le \hat B, 
\end{align}
and, with $\hat\al_t$ denoting the many-body evolution of  observables on $\cF$, 
\begin{align}
	\dG(\al_t(A))=\hat \al_t(\hat A). 
\end{align}
Then, applying $\dG$ on both sides of \eqref{MonoEst} yields the many-body approximate monotonicity estimate
\begin{align}
	\label{MEMB}
	\hat \al_t\del{\hat \A_s(t,\chi)}  \le	{\hat \A_s(0,\chi)}+C{s^{-1}}\sbr{{\hat \A_s(0,\xi)}+(\abs{t}{s^{-n}+C_V)N} },
\end{align}
where $N=\dG(\one)$ is the number operator. 
{See \cite{FLS1, FLS2,LRSZ,LRZ} for related results  based on this technique for quantum many-body systems arising from condensed matter physics.}

\section{Proofs of Theorems \ref{THMRME}--\ref{THMAP}}\label{chapProofs}
In this section, we proved the main results presented in \secref{secIdea}.

We begin with the  precise definition of \eqref{Astchi}.
Fix $c>0$, together with a densely defined self-adjoint operator $\phi$. For each $s>0$, we define a class of observables by functional calculus:
\begin{equation}\label{chi-ts}
	\fullfunction{\A_s}{\Rb\times L^\infty(\Rb)}{\cB(\hf)}{(t,\chi)}{\chi\del{\frac{\phi-c\abs{t}}{s}}}.
\end{equation} 	
For a parameter $0<\delta<1$, we define a class $\cX\equiv \cX_\delta$ as follows: 
\begin{equation}\label{F}
	\begin{aligned}
		\cX
		:=&\Set{\chi\in C^\infty(\R,\Rb_{\ge0})\left|
			\begin{aligned}
				&\supp \chi\subset (0,\infty) ,\,\chi'\ge0
				,\\&\sqrt{\chi'}\in C_c^\infty,\,\supp \chi'\subset (0,\delta)
			\end{aligned}\right.
		}.
	\end{aligned}
\end{equation}
Then,  for any $s,t$, the operator $\A_s(t,\chi),\,\chi\in\cX$ is bounded on $\hf$ and non-negative definite,  with $\norm{\A_s(t,\chi)}\le \norm{\chi}_{L^\infty}$. Typical examples of functions in $\cX$ are suitably smoothed characteristic functions of $\Rb_{\ge0}$ as in \figref{chifig1}.

In what follows, we will use two properties of the space $\cX$, which can be readily verified:
\begin{enumerate}[label=(X\arabic*)]
	\item \label{X1}  If $\xi(x)=\int_0^x w^2(y)\,dy$ for some $w\in C_c^\infty$ with $\supp w\subset (0,\delta)$, then $\xi\in \cX$.

	\item \label{X2}   For any $\xi_1,\,\xi_2\in\cX$ and $c\ge0 $, there exists $\xi\in\cX$ with $\xi\ge\xi_1+c\xi_2$ and $\xi'\ge \xi_1'+c\xi_2'$.
\end{enumerate} 
In principle, the class $\cX$ could be replaced by suitable classes of functions satisfying the abstract properties \ref{X1}--\ref{X2}.

In view of relation \eqref{HeisDrel}, to prove Theorems \ref{THMRME}--\ref{THMAP}, it suffices to derive an estimate for the Heisenberg derivative $D_{H_0}\A_s(t,\chi)$ associated with the free Hamiltonian $H_0$. Thus, in Sections \ref{sec:pfRME} and \ref{sec:pfMonoEst}, we only work with the free evolution and  write $H\equiv H_0$  and 
\begin{align}
	\label{Heis-der}&D A(t)=\frac{\partial}{\partial t}A(t)+i[H, A(t)],
\end{align}
so that, with $\al_t$ denoting the  unitary evolution generated by $H_0$, the Heisenberg equation \eqref{HE} reads
\begin{align}
	\label{dt-Heis}
	\di_t\al_t(A(t))=\al_t(DA(t)).
\end{align}

\subsection{Proof of {\thmref{THMRME}} } \label{sec:pfRME}

Let $\vec\k:=(\kappa_1,\ldots,\kappa_{n+1})$ as in \eqref{commHphi} and set $\delta:=c-\kappa$. Recall in this subsection $\al_t$ denotes the free evolution and $H\equiv H_0$.  

The main result of this section is the following differential operator inequality:

\begin{theorem}\label{thm:RME} 
Suppose   condition \eqref{commHphi} holds for some $n\ge1$.  
	Then, for all $c>\kappa$ and   $\chi\in \mathcal \cX$, there exists a constant
	$C>0$ and functions $\xi_k\in\cX,\,k=2,\ldots, n$ (dropped if $n=1$) depending only on $n,\,\vec\k$, and $\chi$, such that for all $t\in\Rb,\,s>0$, the following operator inequality holds on $\hf$:
	\begin{align}\label{rme}
		&\di_t\al_t\del{\A_s(t,\chi)}\leq -\delta s^{-1}\al_t\del{\A_s(t,\chi')} + \sum_{k=2}^n s^{-k}\al_t\del{\A_s\del{t,\xi_k'}}+ C{s^{-(n+1)}}.
	\end{align}
(The sum in the r.h.s. is dropped if $n=1$.)
\end{theorem}
This theorem is proved at the end of this section. Estimate \eqref{rme}, together with property \ref{X2} and relation \eqref{HeisDrel}, implies \thmref{THMRME}.

\begin{remark}\label{rem:Ham}
	{Identity \eqref{dt-Heis} plays a crucial role in our analysis, and it is precisely in \eqref{dt-Heis} that the Hamiltonian structure of \eqref{1.1} is used. Indeed, for a heat-type equation $\di_t u=-Hu$ with self-adjoint $H$, we have formally, instead of \eqref{dt-Heis}, 
		$$\di_t \al_t(A(t))=\al_t(\di_t A(t) - \Set{H,A(t)}),$$
		where the brace denotes the anti-commutator. The change $[\cdot,\cdot]\to \Set{\cdot,\cdot}$ renders key expansion formulae below unavailable, and thus new machinery is needed to handle heat type equations. We will not seek to pursue this problem presently.}
\end{remark}

We begin with  the following lemma:
\begin{lemma}
\label{lem5.1}
Suppose the assumption of \thmref{thm:RME} holds.
Then there exist $\xi_k=\xi_k(n,\vec\k,\chi)\in\cX,\,k=2,\ldots, n$ (dropped if $n=1$),
together with a constant  $C=C(n,\vec\k,\chi)>0,$  such that the following operator inequality holds on $\hf$:
\begin{equation}
	\label{4.4}
	\begin{aligned}
		&i[H,\A_s(t,\chi)]\le s^{-1}\kappa\A_s(t,\chi')+ \sum_{k=2}^{n}s^{-k}\A_s(t,\xi_k')+ C{s^{-(n+1)}}\quad(t\in\Rb,s>0).
	\end{aligned}
\end{equation}
(The sum in the r.h.s. is dropped if $n=1$.)
\end{lemma}

\begin{proof}
Within this proof, we fix $t$ and write $\A_s(\chi)\equiv \A_s(t,\chi)$.  Also, we set $B_k\equiv\pm i\ad{k}{\phi}{H}$ for $k=1,...,n+1$. (The sign is irrelevant for our argument.) 

1. By condition	\eqref{commHphi}, there exists 	$C=C(n,\vec\k)>0,$ such that\begin{align}\label{BCond}
	\norm{B_k}\le C,\quad k=1,\ldots,n+1.
\end{align} This, together with the definition of $\cX$ (see \eqref{F}), implies that the hypotheses of \lemref{lemA.2}  are satisfied for $\chi\in\cX$, and so there hold the commutator expansion
\begin{align}  [H, \A_s(\chi)]= 
	\sum_{k=1}^n \frac{s^{- k}}{k!}\A_s(\chi^{(k)}) B_k + s^{-(n+1)}R_{n+1},\label{commex'} 
\end{align}
 with some 	$C=C(n,\vec\k,\chi)>0$ such that (c.f.~\eqref{commex}--\eqref{4.A.21})
\begin{align}\label{RemEstEst}
	\norm{R_{n+1}}\le C.
\end{align}
Adding commutator expansion \eqref{commex'} to its adjoint and dividing the result by two, we obtain
\begin{align}
	i[H,\A_s(\chi)]=&\mathrm{I}+\mathrm{II}+\mathrm{III},\label{gvExp}\\
	\mathrm{I}	=&	\frac{1}{2}s^{-1} \left(\A_s(\chi')B_1+B_1^*\A_s(\chi')\right),\label{leading}
	\\	\mathrm{II}=&\frac{1}{2}\sum_{k=2}^{n}\frac{s^{-k}}{k!}\left(\A_s(\chi^{(k)})B_k+B_k^*\A_s(\chi^{(k)})\right),\label{A-sym-exp}
	\\\mathrm{III}=&\frac{1}{2}s^{-(n+1)}\left(R_{n+1}+R_{n+1}^*\right)\label{Rem3},
\end{align}
where the term $\mathrm{II}$ is dropped for $n=1$.

2. We first bound the term $\mathrm{I}$ in line \eqref{leading}.  Let $u:=\sqrt{\chi'}$, which is well defined and lies in $C_c^\infty(\Rb)$ by \eqref{F}. Then, by \eqref{BCond} and \lemref{lemA.2}, expansion \eqref{commex'} also holds for $u$. This expansion, together with the fact that $\Ad_{\phi}^l(B_k)=B_{k+l}$, implies
\begin{align}
	&\quad	\A_s(\chi')B_1+B_1^*\A_s(\chi')\notag
	\\	&=\A_s(u)^2B_1+B_1\A_s(u)^2\notag\\&=2\A_s(u)B_1\A_s(u)+\A_s(u)[\A_s(u),B_1]+[B_1,\A_s(u)]\A_s(u)\nonumber\\
	&=2\A_s(u)B_1\A_s(u)\notag\\&\quad+\sum_{l=1}^{n-1}\frac{s^{-l}}{l!}\left( {\A_s(u)B_{1+l} \A_s(u^{(l)})+\A_s(u^{(l)})B_{1+l}^{*}\A_s(u)}\right)\label{518}\\&\quad+s^{-n}(\A_s(u)\Rem_1+\Rem_1^{*}\A_s(u)),\label{519}
\end{align}
where  line \eqref{518} is dropped for $n=1$ and,  for some 	$C=C(n,\vec\k,\chi)>0$,
\begin{align}\label{Rem1Est}
	\norm{\Rem_1}\le C.
\end{align}

We will bound the terms in \eqref{518}--\eqref{519} using the operator estimate
\begin{align}\label{op-esti}
	\pm (P^*Q+Q^*P)&\leq P^*P+Q^*Q. 
\end{align}
For the terms in line \eqref{518}, we use \eqref{op-esti} with
\begin{align}
	P=\A_s(u),\quad Q:= B_{1+l}\A_s(u^{(l)}),\quad   l=1,\ldots,n-1 , 
\end{align}
yielding 
\begin{align}\label{Exp1}
	&s^{-l}({\A_s(u)B_{1+l} \A_s(u^{(l)})+\A_s(u^{(l)})B_{1+l}^{*}\A_s(u)})\notag\\\le& s^{-l}\del{  \A_s(u)^2+ \|B_{1+l}\|^2(\A_s(u^{(l)}))^2}.
\end{align}
For the remainder terms in \eqref{519}, we apply \eqref{op-esti} with \begin{align}
	P=\A_s(u),\quad Q=\Rem_1,
\end{align} to obtain
\begin{align}\label{Exp2}
	s^{-n}(\A_s(u)\Rem_1+\Rem_1^{*}\A_s(u))&\leq s^{-n}\del{\A_s(u)^{2}+\|\Rem_1\|^2 }.
\end{align}
Combining \eqref{Exp1} and \eqref{Exp2} in \eqref{leading} yields 
\begin{align}\label{Exp4}
	&\mathrm{I}\le  s^{-1} \A_s(u)B_1\A_s(u)\\&\quad+\frac12\sum_{l=1}^{n-1}\frac{s^{-(l+1)}}{l!}\del{  \A_s(u)^2+\notag \|B_{1+l}\|^2(\A_s(u^{(l)}))^2}+\frac12s^{-(n+1)}\|\Rem_1\|^2 .
\end{align}
This bound the term $\mathrm{I}$ \eqref{leading}.

%

3. For $n\ge2$, the term $\mathrm{II}$ in line \eqref{A-sym-exp}  is bounded similarly as in Step 2. For $k=2,...,n$, we take $\theta^k\in C_c^\infty(\Rb)$ with
\begin{align}
	\label{thetaCond}
	\supp\theta^k\subset (0,\delta),\quad \theta^k\equiv 1 \text{ on }\supp\chi^{(k)}.
\end{align}
{We claim that for some bounded operator $\Rem_k=O(1)$, 
	\begin{align}\label{526}
		&s^{-k}\del{\A_s(\chi^{(k)})B_k+B_k^*\A_s(\chi^{(k)})}\notag\\=& s^{-k}\del{\A_s(\chi^{(k)})B_k\A_s(\theta^k)+\A_s(\theta^k)B_k^*\A_s(\chi^{(k)})}+s^{-(n+1)}\Rem_k.
	\end{align}
	For this, it suffices to show that 
	\begin{align}\label{527}
		\A_s(\chi^{(k)})B_k=\A_s(\chi^{(k)})B_{k}\A_s(\theta^k)+ s^{-(n+1-k)}\Rem_k.
	\end{align}
	Using relation \eqref{thetaCond}, commutator expansion \eqref{commex'},  and the fact that $\Ad_{\phi}^l(B_k)=B_{k+l}$, we have
	\begin{align}
		&\quad\A_s(\chi^{(k)})B_k\notag\\&=\A_s(\chi^{(k)})\A_s(\theta^k)B_k\notag\\&=\A_s(\chi^{(k)})B_{k}\A_s(\theta^k)+\A_s(\chi^{(k)})[\A_s(\theta^k),B_k]\nonumber\\
		&=\A_s(\chi^{(k)})B_k\A_s(\theta^k)\notag\\&\quad+\sum_{l=1}^{n-k}\frac{s^{-l}}{l!}\A_s(\chi^{(k)})\A_s((\theta^k)^{(l)})B_{k+l}+s^{-(n+1-k)}\A_s(\chi^{(k)})\Rem_k,\label{531}
	\end{align}
	where the $l$-sum is dropped for $k=n$ and \begin{align}\label{RemkEst}
		\Rem_k\le C,\quad k=2,\ldots,n,
	\end{align}
	for some 	$C=C(n,\vec\k,\chi)>0$.
	
	Since $\theta^k\equiv 1$ on $\supp(\chi^{(k)})$, we have $\supp((\theta^k)^{(l)})\cap\supp(\chi^{(k)})=\emptyset$ for all $l\geq 1$ and so in line \eqref{531}, 
	\begin{align*}
		\A_s(\chi^{(k)}) \A_s((\theta^k)^{(l)})B_{k+l}=0,\quad l=1,\ldots n-k.
	\end{align*}
	Estimate \eqref{527} follows from here. Thus we conclude claim \eqref{526}.
	
	Now, we apply estimate \eqref{op-esti} on the first term on the r.h.s. of \eqref{526} with \begin{align}
		P=\A_s(\chi^{(k)}),\quad Q=B_k\A_s(\theta^k),
	\end{align} and then sum over $k$ to obtain
	\begin{align}\label{Exp3}
		\mathrm{II}	\leq&\frac12\sum_{k=1}^{n-1} \frac{s^{-k}}{k!}\left((\A_s(\chi^{(k)}))^2+\|B_k\|^2(\A_s(\theta^k))^2\right)+\frac12s^{-(n+1)}\norm{\Rem_k}^2.
	\end{align}
	This bounds the term  $\mathrm{II}$ in line \eqref{A-sym-exp}.
}


4.	Plugging \eqref{Exp4}, \eqref{Exp3} back to \eqref{gvExp} and using bounds \eqref{BCond}, \eqref{RemEstEst}, \eqref{Rem1Est}, and \eqref{RemkEst}, we find that   for some 	$C=C(n,\vec\k,\chi)>0$, 
\begin{align} 
	&i[H,\A_s(\chi)] \le   s^{-1}\A_s(u)B_1\A_s(u)\label{535}\\&+ C\sum_{k=2}^{n}{s^{-k}} \del{  \A_s(u)^2+(\A_s(u^{(k-1)}))^2+\A_s(\chi^{(k)}))^2+(\A_s(\theta^k))^2} +Cs^{-(n+1)}.\notag
\end{align}
Now, for $k=2,\ldots,n$, we choose, with $C$, $u$, $\theta^k$ from \eqref{535},
\begin{align}
	&w_k\in C_c^\infty, \quad  \supp w_k\subset (0,\delta),\notag\\  w_k^2\ge &C \del{u^2+ (u^{(k-1)})^2+(\chi^{(k)})^2+(\theta^k)^2},
	\label{wChoice}\end{align}
which is possible since the r.h.s.~of \eqref{wChoice} is supported in $(0,\delta)$ by construction. Then the function
\begin{align}\label{xiChoice}
	\xi_k(x):=\int_0^x w_k^2(y)\,dy
\end{align}
lies in $\cX$ by identity \ref{X1}. Thus, by \eqref{535}, the desired estimate \eqref{4.4} holds with the choice of $\xi_k$ from \eqref{xiChoice}.
This completes the proof of \lemref{lem5.1}.
\end{proof}

\begin{proof}[Proof of \thmref{thm:RME}]
To prove estimate \eqref{rme}, we first apply the differential identity \eqref{dt-Heis} with $A(t)=\A_s(t,\chi)$ for each $s,\,\chi$. This yields
\begin{equation}\label{121}
	\di_t\al_t(\A_s(t,\chi))=\al_t(\di_t\A_s(t,\chi))+\al_t(i[H,\A_s(t,\chi)]).
\end{equation}
By definition \eqref{chi-ts}, we find
\begin{align} \label{eq:deriv}
	\di_t\A_s(t,\chi)=-s^{-1}c\,  \A_s(t,\chi').
\end{align}
By estimate \eqref{4.4}, we find 
\begin{align}\label{122}
	i[H,\A_s(t,\chi)] &\le  s^{-1}\kappa  \A_s(t,\chi') +  \sum_{k=2}^ns^{-k}\A_s(t,\xi_k')+C{s^{-(n+1)}},
\end{align}
where $C=C(n,\vec\k,\chi)>0$ and the second term in the r.h.s. is dropped for $n=1$.  	Plugging \eqref{eq:deriv} and \eqref{122} back to \eqref{121} and using the positive-preserving property of evolution $\al_t$  yields \eqref{rme}.
\end{proof}

%

\subsection{Proof of \thmref{THMME}}\label{sec:pfMonoEst}
Recall in this subsection $\al_t$ denotes the free evolution and $H\equiv H_0$. 
Our main result is the following:

	\begin{theorem} \label{THMME'}
		Suppose \eqref{rme} holds.  
		Then there exist 
		$C>0$ and a function  $\xi\in\cX$ (dropped if $n=1$) depending only on $n,\,\vec\k$, $\chi$, and $\delta$ such that for all $t\in\Rb,\,s>0$, the following operator inequality holds on $\hf$:
		\begin{align}\label{336}
\al_t\del{\A_s(t,\chi)}  \le	{\A_s(0,\chi)}+ {s^{-1}}{\A_s(0,\xi)}+C\abs{t}s^{-(n+1)}.
		\end{align}
		(The second term on the r.h.s. is dropped for $n=1$.)

	\end{theorem}

Using Estimate \eqref{336}, together with relation \eqref{HeisDrel} and the $L^1$ bound on $G(t)$ from assumption \eqref{commVphi}, we arrive at \thmref{THMME}.

\begin{proof}[Proof of \thmref{THMME'}]

Within this proof, we fix $s>0$ and	all constants $C>0$ depend only on $n$, $\chi$,  $\vec\k$, and $\delta=c-\kappa$. For simplicity, below we consider the case for $t\ge0$. For negative times the argument is similar.

1. For ease of notation, for any function $f\in L^\infty$, we write
\begin{align}\label{2.2}
f[t]:= \al_t(\A_s(t,f)).\end{align}
Note in particular that $f[0]\equiv \A_s(0,f)$.

To begin with, we claim the following holds: 
There exist $\tilde \xi_k\in \cX$, $2\le k\le n$ (dropped for $n=1$), depending only on $n,\vec\k,\chi$, and $\delta$,
such that for all  $t\ge0\,,s>0$,
\begin{align}
&\int_0^t \chi'[t]dr  \le C  \del{s\chi[0]
	+ \sum_{k=2}^n s^{-k+2}\tilde \xi_k[0]+   ts^{-n}}, 
\label{propag-est31} 
\end{align}
where the sum is dropped if $n=1$.



To prove \eqref{propag-est31}, we bootstrap the recursive monotonicity estimate \eqref{rme}. For each fixed $s$, integrating formula \eqref{dt-Heis} with $A(t)\equiv \A_s(t,\chi)$ in $t$ gives
\begin{align} \label{eq-basic}  
\chi[t]-\int_0^t \di_r\chi[r]\,dr= \chi[0].
\end{align}
We apply  inequality \eqref{rme} to the second term on the l.h.s. of \eqref{eq-basic} to obtain, after transposing the $O(s^{-1})$-term,
\begin{align} \label{propag-est2} 
&\chi[t]+ s^{-1}\delta\int_0^t\chi'[r]\, dr 
\le  \chi[0]+ \sum_{k=2}^ns^{-k}\int_0^t  \xi_k'[r]\,dr + C{t s^{-(n+1)}},
\end{align}
where $\delta=c-\kappa$, $\xi_k=\xi_k(n,\vec\k,\chi)\in\cX$,
and the second term in the r.h.s. is dropped for $n=1$.

Since $s\,,\delta>0$, estimate \eqref{propag-est2} implies, after dropping $\chi[t]$ on the l.h.s., which is non-negative-definite due to the positive-preserving property of the evolution $\al_t$, and multiplying both sides by $s\delta^{-1}>0$, that
\begin{align}
\int_0^t\chi'[r]\, dr 
\le \frac{1}{\delta} 
\del{s \chi[0]+   \sum_{k=2}^ns^{-k+1}\int_0^t \xi'_k[r]\,dr+ Ct s^{-n}}, \label{propag-est3} 
\end{align}
where the second term in the r.h.s. is dropped for $n=1$. 

If $n=1$, then \eqref{propag-est3} gives \eqref{propag-est31}. If $n\ge2$, we proceed  to apply \eqref{propag-est3} to the term $\int_0^t \xi'_2[r]\,dr$ up to $(n-1)$-th order to get
\begin{equation}\label{223}
\int_0^t \xi_2'[r]\,dr\le \frac{1}{\delta} 
\del{s \xi_2[0] +   \sum_{k=2}^{n-1}s^{-k+1}\int_0^t  \eta_k'[r]\,dr+ Ct s^{-(n-1)}},
\end{equation}
where the sum is dropped for $n=2$ and, for $n\ge3$,	$$\eta_k=\eta_k(n,\vec\k, \xi_2)\in\cX,\quad k=2,\ldots, n-1.$$ 
Plugging \eqref{223} back to \eqref{propag-est3}, we find
\begin{align}
\int_0^t\chi'[r]\, dr 	\le& \frac{1}{\delta} \del{s
\chi[0]+  \frac{1}{\delta}\xi_2[0] + \sum_{k=3}^n s^{-k+1} \int_0^t \rho_k'[r]\,dr + \del{1+\frac1\delta} C t s^{-n}}, \label{propag-est33} 
\end{align}
where the third term in the r.h.s.~is dropped for $n=2$ and the functions $\rho_k\in\cX$, $\rho_k'\ge \xi_k'+\tfrac1\delta\eta_k'$  for $k=3,\ldots, n$ (see \ref{X2}).
Bootstrapping this procedure, we arrive at \eqref{propag-est31} for $n\ge2$.

2. Now we use \eqref{propag-est31} to derive the desire estimate \eqref{MonoEst}.

Dropping the second term in the l.h.s. of \eqref{propag-est2}, which is non-negative since  $\delta>0$ and $ \chi'[r]\ge0$ for all $r$, we obtain
\begin{align} \label{128}
 \chi[t]\le  \chi[0]+ \sum_{k=2}^ns^{-k}\int_0^t  \xi_k'[r]\,dr  + C{ t s^{-(n+1)}},
\end{align}
where the second term is dropped for $n=1$ (in which case we are done). If $n\ge2$, then for each $k=2,\ldots,n$, we apply estimate \eqref{propag-est31} to the $k$-th summand in the second term in the r.h.s. of \eqref{128}, with remainder expanded to $(n-k+1)$-th order. This way we obtain
\begin{align}\label{e3442}
 \chi[t]\le  \chi[0]+ C\sbr{\sum_{k=2}^n\sum_{l=2}^{n-k}s^{-(k-1)}\del{ \tilde \xi_k[0]+s^{-(l+k-2)} \tilde\xi_{k,l}[0]}}+C  ts^{-(n+1)}.
\end{align}
where the $k$-sum is dropped for $n=1$,  the $l$-sum is dropped if $n-k\le1$, and $C$, $\tilde \xi_{k,l}$ are chosen according to \eqref{propag-est31}.

Using property \ref{X2}, we can choose $\xi\in \cX$ such that for $C,\,\tilde \xi_k,\,\tilde \xi_{k,l}$ as in \eqref{e3442},
\begin{align}\label{xiChoice'}
\xi \ge C\DETAILS{M_\delta} \sbr{\sum_{k=2}^n\sum_{l=2}^{n-k}\del{\tilde \xi_k[0]+ \tilde\xi_{k,l} [0]}}.
\end{align}
With this choice of $\xi$, we conclude the desired estimate, \eqref{336}, from \eqref{e3442}. This completes the proof of \thmref{THMME'}. 
\end{proof}

\subsection{Proof of Theorem \ref{THMAP}} 
\label{sec:pfMVE}


Recall that $\phi$ is a densely defined self-adjoint operator on $\hf$ and  $P_{{ a}}$ denotes the spectral cutoff operator defined in \eqref{Pdef}. Our goal now is to choose a  function $s=s(t)$ s.th.~the geometric inequalities \eqref{chi-0s-est'}--\eqref{chi-ts-est'} hold, whereby eliminating the adiabatic parameter $s$ and the ASTLOs from \eqref{MonoEst} so as to conclude the desired estimate \eqref{eLCgen}.

Our main result is the following proposition:
\begin{proposition}\label{propGeo}
Let $\delta,\,c'>0$. For functions $f(t)>c'\abs{t} $ and  $\eta\in C^1\cap L^\infty(\Rb,\Rb_{\ge0})$ with
\begin{equation}\label{41}
	\eta\not\equiv0,\quad \supp \eta\subset(0,\infty),\quad \supp\eta'\subset (0,\delta),
\end{equation}
let
\begin{align}\label{Adef}
	s:=\delta^{-1} (f(t)-c'\abs{t}),\quad 	\A(t,\eta):=\eta(s^{-1}(\phi-c'\abs{t})) .
\end{align}
Then the following estimates hold:
\begin{align}\label{chi-0s-est} 
	&\norm{\eta}_{L^\infty}^{-1} \A(0,\eta)  \le  P_{{0}},\\
	\label{chi-ts-est}&P_{{f(t)}} \le \norm{\eta}_{L^\infty}^{-1} \A(t,\eta)\quad(t\ne0).
\end{align}

\end{proposition}
\begin{proof}
First, by \eqref{41}, we have    $\supp \eta\big(\frac{\cdot}{s}\big) \subset (0,\infty)$ for $s>0$. This implies 
\begin{align}
	\norm{\eta}_{L^\infty}^{-1}\A(0,\eta)\equiv \norm{\eta}_{L^\infty}^{-1}\eta\del{\phi/s}\le\theta(\phi)\equiv P_{{0}},\label{447'}
\end{align}
where 	$\theta:\Rb\to\Rb$ is the characteristic function of the half-line  $(0,\infty)$ (see \figref{fig:f0}). Thus  
\eqref{chi-0s-est} follows.

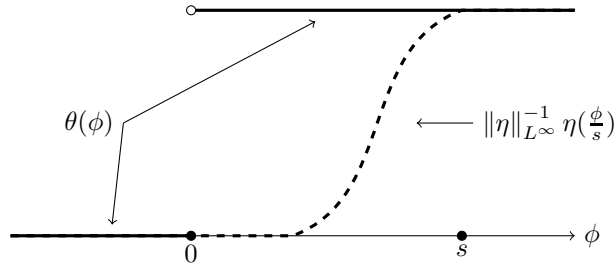
\begin{figure}[H]
	\centering
	\begin{tikzpicture}[scale=3]
		\draw [->] (-.5,0)--(2,0);
		\node [right] at (2,0) {$\phi$};
		\node [below] at (.3,0) {$0$};
		\draw [fill] (.3,0) circle [radius=0.02];
		
		
		\node [below] at (1.5,0) {$s $};
		\draw [fill] (1.5,0) circle [radius=0.02];

		\draw [very thick] (-.5,0)--(.3,0);
		\draw [very thick] (.3,1)--(2,1);				
		\filldraw [fill=white] (.3,1) circle [radius=0.02];
		
		\draw [dashed, very thick] (-.5,0)--(.75,0) [out=20, in=-160] to (1.5,1)--(2,1);

		\draw [->] (1.55,.5)--(1.3,.5);
		\node [right] at (1.55,.5) {$\norm{\eta}_{L^\infty}^{-1}\eta(\tfrac\phi s)$};
		
		\draw [->] (0,.5)--(.85,.95);
		\draw [->] (0,.5)--(-.05,.05);
		\node [left] at (0,.5) {$\theta(\phi)$};
		
	\end{tikzpicture}
	\caption{Schematic diagram illustrating \eqref{447'}}
	\label{fig:f0}
\end{figure}

Next, again by \eqref{41}, we have $\norm{\eta}_{L^\infty}^{-1}\eta (\mu)\equiv 1$ for $\mu> \delta$ and so, by definition \eqref{Adef}, \begin{align}\label{450'}
	\norm{\eta}_{L^\infty}^{-1}\A(t,\eta)\equiv \norm{\eta}_{L^\infty}^{-1}\eta \del{\delta\frac{\phi-c'\abs{t}}{f(t)-c'\abs{t}}} \equiv \one,
\end{align} on the subspace $\Ran P_{{f(t)}}$.
Since   $P_{{f(t)}}\equiv \theta(\phi-f(t))$, estimate \eqref{450'} implies
\begin{align}\label{448'}
	\norm{\eta}_{L^\infty}^{-1}\A(t,\eta)\ge \theta(\phi-f(t)),
\end{align}see \figref{fig:f}. Thus
\eqref{chi-ts-est} follows.

\begin{figure}[H]
	\centering
	\begin{tikzpicture}[scale=3]
		\draw [->] (-.5,0)--(2,0);
		\node [right] at (2,0) {$\phi$};
		
		\node [below] at (0,0) {$c't$};
		\draw [fill] (0,0) circle [radius=0.02];
		
		\node [below] at (1.5,0) {$f(t)$};
		\draw [fill] (1.5,0) circle [radius=0.02];

		\draw [very thick] (-.5,0)--(1.5,0);
		\draw [very thick] (1.5,1)--(2,1);
		\draw [dashed, very thick] (-.5,0)--(.1,0) [out=20, in=-160] to (.65,1)--(2,1);
		\filldraw [fill=white] (1.5,1) circle [radius=0.02];

		\draw [->] (-.1,.5)--(.3,.5);
		\node [left] at (-.1,.5) {$\norm{\eta}_{L^\infty}^{-1}\eta(\tfrac{\phi-c'\abs{t}}{s})$};
		
		\draw [->] (2,.5)--(1.75,.95);
		\draw [->] (2,.5)--(1.25,.05);
		\node [right] at (2,.5) {$\theta(\phi-f(t))$};
		
	\end{tikzpicture}
	\caption{Schematic diagram illustrating \eqref{448'}.}
	\label{fig:f}
\end{figure}
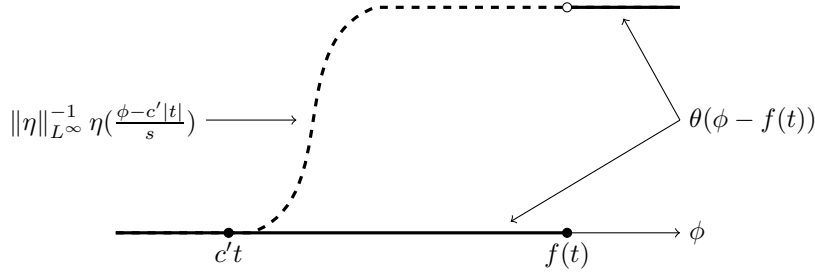
This completes the proof of Proposition \ref{propGeo}.
\end{proof}

We now use \propref{propGeo} and Theorem \ref{THMME} to prove \thmref{THMAP}.

First, for $c>\kappa$ as in the statement of \thmref{THMAP}, we set
\begin{align}\label{deltaChoice}
\delta:=\frac13(c-\kappa)>0,\quad c':=\kappa+\delta. 
\end{align}
Fix any $\chi\in\cX$. 
We apply \thmref{THMME} with $c'>\kappa$ to get  a constant
$C>0$ and a function $\xi\in\cX$ such that 
\begin{align}\label{MonoEst'}
\al_t\del{\A_s(t,\chi)}  \le 	 {\A_s(0,\chi)}+Cs^{-1} {\A_s(0,\xi)}+Cs^{-1}\del{s^{-n} \abs{t}+C_V}.
\end{align}
Next, we apply \propref{propGeo} with 
\begin{align}
f(t):=c\abs{t}>c'\abs{t},\quad s:=\delta^{-1}(c-c')\abs{t}>\abs{t}, 
\end{align}
where the inequalities are ensured by the choice \eqref{deltaChoice}.
The function $\chi$ clearly satisfies condition \eqref{41}. If the function $\xi\not\equiv 0$ in \eqref{MonoEst'}, then   $\xi$ also satisfy \eqref{41}. (If $\xi\equiv 0$ then we drop the second term in the r.h.s. of \eqref{MonoEst'}). Hence,
applying \eqref{chi-0s-est}--\eqref{chi-ts-est} with $\eta=\chi,\xi$ and $\A\equiv \A_s$ as in \eqref{Adef}, we conclude the desired estimate, \eqref{eLCgen}, from estimate \eqref{MonoEst'}.

This completes the proof of \thmref{THMAP}.\qed

\section{Applications to nonlocal dispersive equations
}\label{chapMVENonLoc}

In this section, we apply the general localization theory laid out in Section \ref{secRes} to study a large class of nonlocal dispersive evolution models.

We 	consider the following nonlocal non-autonomous Schr\"odinger equation:
\begin{equation}\label{1.1}
	i\di_t\psi= H(t)\psi.
\end{equation}
Here  $\psi=\psi(\cdot,t),\,t\in\Rb$ is a differentiable path of vectors in the Hilbert space $\hf:=L^2(\Rb^d,\Cb),\,d\ge1$. The Hamiltonian $H(t)=H_0+V(t)$ consists of a nonlocal part
\begin{equation}\label{1.2}
	H_0[\psi](x)=p.v.\int_{y\in\Rb^d} (\psi(x)-\psi(y))K(x,y),
\end{equation}
for some symmetric (and possibly singular) integral kernel $K$ with $K(y,x)=\overline{K(x,y)}$, together with a time-dependent potential $V(t).$ 

As a standing assumption, we assume that $H_0$ is self-adjoint on a dense domain $\cD\equiv \cD(H_0)\subset \hf$ and $V(t)$ is uniformly bounded for all $t$. Consequently, $H(t)$ is self-adjoint on $\cD$ and so, by standard perturbation theory, admits bounded propagator $U(t,s)$ with $t,s\in\Rb$ (see e.g.~\cite[Theorem 25.32]{GS}). 

Our main technical assumption is the following: For some integer $n\ge1$ and function $\phi\in \dot{W}^{1,\infty}$ (i.e., weakly differentiable with $\grad \phi\in L^\infty),$  the operators
\begin{equation}\label{BpDef}
	B_p[f](x):=\int_{y\in \Rb^d} 	{K(x,y)}(\phi(x)-\phi(y)) ^pf(x)
\end{equation}
satisfy, for $p=1,\ldots, n+1$ and some $\kappa_p>0$, \begin{align}\label{k-cond0}
	\norm{B_p}_{L^2\to L^2}\le \kappa_p  . 
\end{align}


We show in Appendix \ref{sec:A} that condition \eqref{k-cond0} amounts to the main technical condition \eqref{commHphi} in the general theory (see \secref{secMainRes}) and that a sufficient condition for \eqref{k-cond0} is
\begin{equation}\label{k-cond1}
	\max_{1\le p\le n+1}\sup_{x\in\Rb^d}\int_{y\in \Rb^d} 	\abs{K(x,y)}\abs{x-y}^p <\infty.
\end{equation}
Typical examples of the form \eqref{1.2} satisfying \eqref{k-cond1} include the nonlocal diffusion operators 
\begin{equation}\label{1.3}
	H_0= 1- J*,
\end{equation}
where $J$ is a   radial function with profile satisfying \begin{equation}\label{j-cond}
	\sup_{1\le p\le n+1} \int_0^\infty r^{p+d-1}\abs{J(r)}\,dr<\infty,
\end{equation}
e.g.,  $J(x)=(1+\abs{x}^2)^{-a/2}$ with $a>d+n+1$. By interpolation, mild singularity is allowed at $0$, e.g., $J(x)=O(|x|^{-b})$ with $b<d+1$. 

Condition \eqref{k-cond0} are also verified by certain fractional differential operators.
{In his seminal work \cite{MR0177312}, Calder\'on proved that \eqref{k-cond0} holds for $p=1$ and  $H_0=(-\Lap)^{1/2}$, or equivalently, 
	\begin{align}
		K(x,y)=\frac{1}{\abs{x-y}^{d+1}}.
	\end{align}
	The boundedness of commutators of more general singular integral operators and fractional elliptic operators are subsequently established in \cite{MR0412721, MR0358205,MR0763911, MR3286493,MR3547014,MR4443495,MR3555319}, among many others, under various conditions on $H_0$, $K$ and for various classes of functions $\phi$ (typically belonging to $\dot W^{1,\infty}$ or BMO). As the scheme below indicates, boundedness of singular integral operators of the form \eqref{BpDef} would lead to similar propagation estimates in the corresponding dynamical models.
}


Evolution equations involving nonlocal operators of the form \eqref{1.2} have received  much research attention in recent years. 
For recent results concerning evolution equations involving \eqref{1.3} subject to similar conditions as \eqref{j-cond}, see e.g.~\cite{MR2257732,MR2542582,MR3285829,MR3289358,MR3440113,MR4182983} and, for applications to natural sciences, \cite{MR3469920,MR4409816}, as well as the references therein. For regularity theory of nonlocal evolution equations, see \cite{MR3626038,MR3771838,MR3959442}. For an excellent recent review on  nonlocal diffusion operators with integrable kernels, see \cite{MR4187861}. 

Note however that all of the cited works above are concerned with, instead of Hamiltonian evolution equation as in \eqref{1.1}, gradient flows of the form $\di_t\psi=-H\psi$ with $H$ of the form \eqref{1.2}.  This distinction should be made clear since the Hamiltonian structure of \eqref{1.1} is used crucially in proving the recursive monotonicity estimate \eqref{RMEproto} for $\A_s(t,\chi)$ (wherefore in all us results in \secref{secRes} as well). See \remref{rem:Ham} below for a discussion.

Eq.~\eqref{1.1} arises, among others, from the study of nonlinear nonlocal Schr\"odinger (NLS) equations  of the form
\begin{equation}\label{NLS}
	i\di_t \psi = H_0\psi+W\psi+f(\abs{\psi}^2)\psi,\quad f\in C(\Rb_{\ge0},\Rb),
\end{equation}
where $W$ is  a bounded  external potential (possibly time-dependent).
Eq.~\eqref{NLS} has a Hamiltonian structure inherited from the nonlocal generalization of the Ginzburg-Landau free-energy functional in the presence of external potential:
$$E(\psi)=\frac14\iint K(x,y)\abs{\psi(x)-\psi(y)}^2 + \int W(x)\abs{\psi(x)}^2+F(\abs{\psi(x)}^2),\quad F'=f.$$
Indeed, if  $\psi^{(0)}_t\in L^\infty\cap L^2$ solves \eqref{NLS}, then $\psi^{(0)}_t$ satisfies \eqref{1.1}  with $V(t):=W+f(\lvert{\psi^{(0)}_t}\rvert)$ bounded for all $t$. 
This convolution-type model for phase transitions was proposed in \cite{MR1463804} and the associated $L^2$-gradient flow (the nonlocal Allen-Cahn equation) has been studied in \cite{MR1463804,MR1712445,MR1933014,MR2257732,MR2542582}. See \cite[Sect.~1]{MR1712445} for a discussion on the connection between $E(\psi)$ above and the classical Ginzburg-Landau energy functional.

\subsection{Results}\label{secApplRes}

%
%
Under the standing assumption, the evolution of a state $\psi_s\in\cD$ from time $s$ according to \eqref{1.1} is given by
\begin{equation}\label{U}
	\psi_t= U(t,s) \psi_s,
\end{equation}
where $U(t,s),\,s,t\in\Rb$ is the propagator for $H(t)=H_0+V(t)$ in \eqref{1.1}. 
The evolution of an observable $A$, dual to the evolution of states $\psi_s\mapsto \psi_t=U(t,s)\psi_s$  w.r.t.~the coupling $(A,\psi_s)\mapsto\inn{\psi_s}{A\psi_s}$, is given by
\begin{equation}\label{1.5}
	\al_{t,s}(A):=U(t,s)^*AU(t,s),
\end{equation}
where $U(t,s)^*$ is the backward propagator.



For  $\psi_0\in\cD$ and $A\in\cB(\hf)$, we denote by $\psi_t=U(t,0)\psi_0$ and $\al_t(A)=\al_{t,0}(A)$ the evolution of states and observables, respectively. Let  $\kappa=\kappa_1$ and $\vec\kappa=(\kappa_1,\ldots,\kappa_{n+1})$ be as in \eqref{k-cond0}. Our main result in this section is the following:
\begin{theorem}\label{thmLocGen}  
	Suppose \eqref{k-cond0} holds for $n\ge1$ and $\phi\in \dot{W}^{1,\infty}$.   Then, for every $c > \kappa$,
	there exists $C=C(n,c,\norm{\grad\phi}_{L^\infty},\vec\kappa)>0$ such that for any function $f(t)> c\abs{t}$ and $t\ne 0$, 
	\begin{align}
		\label{MVE0}\norm{\1_{\Set{x\mid\phi(x)>f(t)}}\psi_t}^2
		\le& 
		(1+C(f(t)-c\abs{t})^{-1})\norm{\1_{\Set{x\mid\phi(x)>0}}\psi_0}^2 \notag\\&+C\abs{t}(f(t)-c\abs{t})^{-(n+1)} \norm{\psi_0}^2.
	\end{align}
\end{theorem}
\begin{proof} 
	We derive estimate \eqref{MVE} as a consequence of Thms.~\ref{THMRME}--\ref{THMME} and \propref{propGeo}. 
	Fix 
	$t\ne0$ and 
	$\chi\in\cX$ (see \eqref{F}) with
	$\chi(\mu)\equiv 1$ for $\mu\ge1$.  	Below, all estimates are independent of these parameters.

	First, 
	we verify the assumptions of \thmref{THMRME}.
	Since $H=H_0+V$ in \eqref{1.1} with $[V,\phi]=0$, the evolution condition \eqref{HeisDrel} is satisfied with $H_0$ given by \eqref{1.2}. By \lemref{lemA.1}, the Hamiltonian $H_0$ from \eqref{1.2} and $\phi$ verify the commutator condition \eqref{commHphi}, with $\kappa_p$ depending on $\Lip(\phi)$. We have shown that the assumptions of \thmref{THMRME} hold. Thus, by Thms.~\ref{THMRME}--\ref{THMME}, estimate \eqref{MonoEst} holds.

	Next, define $s=s(t):=f(t)-c\abs{t}>0$ and denote by $\A(t,\chi)\equiv \A_s(t,\chi)$ the ASTLOs from  \eqref{chi-ts} with this choice of $s$. Then, by estimate \eqref{MonoEst}, there exists a constant
	$C>0$ and a function $\xi \in\cX$  such that
	\begin{align}\label{MonoEst''}
		\wtdel{\A(t,\chi)} \le&	\wndel{\A(0,\chi)}\notag +(f(t)-c\abs{t})^{-1}\wndel{\A_s(0,\xi)}\\&+C\abs{t}(f(t)-c\abs{t})^{-(n+1)}\norm{\psi_0}^2.
	\end{align}
	
	Lastly, we use  Proposition \ref{propGeo}. The function $\chi$ clearly satisfies condition \eqref{41}. If the function $\xi\not\equiv 0$ in \eqref{MonoEst''}, then $\xi$ also satisfy \eqref{41}. (If $\xi\equiv 0$ then we drop the second term in the r.h.s. of \eqref{MonoEst''}). Hence,
	applying \eqref{chi-0s-est}--\eqref{chi-ts-est} with $\eta=\chi,\xi$ in \eqref{MonoEst''} and using that $P_{{a}}\equiv \1_{\Set{x\mid\phi(x)>a}}$ for all $a>0$ (see \eqref{Pdef}),  we conclude the desired estimate, \eqref{MVE}, from estimate \eqref{MonoEst''}. This completes the proof of Theorem \ref{thmLocGen}.
\end{proof}

Theorem \ref{thmLocGen} grants control over the localization of states $\psi_t$ w.r.t.~to a fixed reference geometry, $\phi(x)$, and a height function $f(t)$.  The growth rate of $f(t)$ in turn determines the decay estimate of the probability leakage as in \eqref{MonoEst''}.

Specifically, let $X\subset \Rb^d$ and $d_X(x)=\inf_{y\in X}\abs{x-y}$. Taking  $\phi=d_X$ and using the facts that $\norm{\grad d_X}_{L^\infty}\le1$,  $\1_{\Set{d_X>c\abs{t}}}\equiv \1_{X_{c\abs{t}}^\cp}$, 
we conclude from  \thmref{thmLocGen} that
\begin{corollary}[Localization of scattering states]   
	Suppose \eqref{k-cond0} holds for $n\ge1$ and $\phi=d_X$.   Then, for every $c > \kappa$,
	there exists $C=C(n,c,\vec\kappa)>0$ such that for all subset $X\subset \Rb^d$, functions $f(t)> c\abs{t}$, and $t\ne 0$, 
	\begin{equation}
		\label{MVE}\norm{\one_{X_{f(t)}^\cp}\psi_t}^2
		\le 
		(1+C(f(t)-c\abs{t})^{-1})\norm{\one_{X^\cp}\psi_0}^2 +C\abs{t}(f(t)-c\abs{t})^{-(n+1)} \norm{\psi_0}^2.
	\end{equation}
	
\end{corollary}

To see that \eqref{MVE} controls the localization of evolving states according to \eqref{1.1},  fix $\eps>0$ and define $f(t)=(c+\eps)\abs{t}$. Assuming the initial condition $\psi$ at $t=0$ is localized in $X$ in the sense that $\norm{\one_{X^\cp}\psi_0}\le \eps$, we conclude from \eqref{MVE} that $\norm{\one_{X_{ct}^\cp}\psi_t}_{L^2}^2\ls \eps+ \abs{t}^{-1}+ \eps^{-(n+1)} \abs{t}^{-n}$ for all $t\ne0$.

As a  consequence of the localization estimate \eqref{MVE}, we have the following a priori estimate on the propagation speed of traveling wave solutions to the nonlinear nonlocal Schr\"odinger equation \eqref{NLS}:
\begin{corollary}\label{corSoliton}Suppose \eqref{k-cond1} holds for $n\ge1$.
	Suppose $\psi_t \in L^2\cap L^\infty,\,t\ge0$ solves the NLS equation \eqref{NLS} and $ \psi_t  = U(\cdot-\beta t)$ for some fixed velocity $\beta\in\Rb^d$ and profile $U$ with the following property:
	There exists a bounded subset $X\subset \Rb^d$ such that $\norm{ \one_{X^c} U}^2<\norm{U}^2/2$.  Then $\abs{\beta}\le \kappa$. 
\end{corollary}
\begin{proof}
	
	Since $\psi_t$ solves \eqref{1.1}, by freezing coefficients, $\psi_t$ satisfies \eqref{MVE} and therefore we have   \begin{equation}\label{1.30}
		\norm{\one_{X_{ct}^\cp}U(x-\beta t)}^2\le \norm{U}^2/2 + C t^{-n},
	\end{equation}
	for all $c>\kappa$.  Suppose now $\abs{\beta}>\kappa$. Then, on the one hand, we can choose $c\in(\kappa,\abs{\beta})$ such that \eqref{1.30} holds. On the other hand, since $c<\abs{\beta}$, there is a large $T\gg1$ depending only on $\abs{\beta}-c$ and $\diam (X)$ such that \begin{equation}\label{130}
		\norm{\one_{X_{ct}^\cp}U(\cdot -\beta t)}^2\ge\norm{\one_{X} U}^2>\norm{U}^2/2 
	\end{equation} for all $t\ge T$ (see \figref{fig:U}). This is a contradiction to \eqref{1.30}. 
	\begin{figure}[H]
		\centering
		\begin{tikzpicture}[scale=.8]
			\draw  plot[scale=.52,smooth, tension=.7] coordinates {(-3,0.5) (-2.5,2.5) (-.5,3.5) (1.5,3) (3,3.5) (4,2.5) (4,0.5) (2.5,-2) (0,-1.5) (-2.5,-2) (-3,0.5)};
			
			\draw[dashed]  plot[shift={(5,0)}, scale=.52,smooth, tension=.7] coordinates {(-3,0.5) (-2.5,2.5) (-.5,3.5) (1.5,3) (3,3.5) (4,2.5) (4,0.5) (2.5,-2) (0,-1.5) (-2.5,-2) (-3,0.5)};
			
			\draw  plot[shift={(-0.2,-0.25)}, scale=.76,smooth, tension=.7] coordinates {(-3,0.5) (-2.5,2.5) (-.5,3.5) (1.5,3) (3,3.5) (4,2.5) (4,0.5) (2.5,-2) (0,-1.5) (-2.5,-2) (-3,0.5)};
			
			\node  at (-0.7,1) {$X$};
			
			\draw [fill] (0.5,0.5) circle [radius=0.05];
			\node[below]  at (0.5,0.5) {$U(\cdot)$};
			
			\draw [fill] (5.5,0.5) circle [radius=0.05];
			\node[below]  at (5.5,0.5) {$U(\cdot-\beta T)$};	 		
			\draw [->] (0.5,0.5)--(5.45,0.5);
			
			\node[below]  at (5.5,-1.5) {$X^\cp_{ct}$};	 		
			
			
			\node [below] at (2.32,0) {$\underbrace{}_{ct}$};
		\end{tikzpicture}
		\caption{Schematic diagram illustrating relation \eqref{130}.}
		\label{fig:U}
	\end{figure}
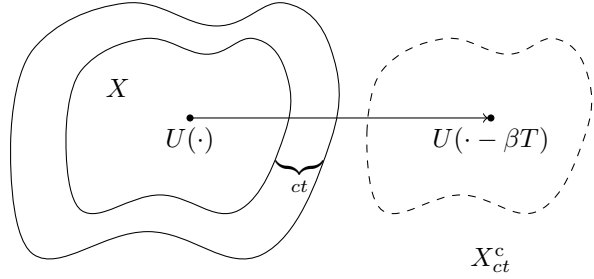
	The proof of \corref{corSoliton} is complete.
\end{proof}

\section{Technical lemmas}\label{chapPre}

\subsection{Remainder estimates}\label{secRemEst}
In this section and the next one, we present some estimates and commutator expansions, first derived in \cite{SigSof}  and then improved in \cite{HunSig3, Skib} etc.  Below, we adapt some of the arguments from \cite{HunSig3} and results from \cite{MVBvNL}.

{Throughout this section we fix an integer $\nu\ge0$.} For integers $p\ge 0$ and smooth functions $f\in C^{\nu+2}(\Rb)$, we define a weighted norm 
\begin{align}\label{pnNorm}
	\cN(f,p):=\sum_{m=0}^{\nu+2}\int_\Rb \br{x}^{m-p-1}\abs{f^{(m)}(x)}\,dx.
\end{align}
Note that
\begin{align}\label{Norder}
	p\le p'\implies \cN(f,p')\le \cN(f,p),
\end{align}
and we have the following property:
\begin{lemma}\label{lemNfin}
	Let $p\ge0$ be an integer. Suppose $f\in C^{\nu+2}(\Rb)$ and {there exist  $C_0,\,\rho>0$ such that for $ m=0,\ldots, \nu+2$, 
		\begin{align}
			\label{fCond}
			\norm{	\br{x}^{m-p+\rho}	f^{(m)}(x)}_{L^\infty}\le C_0 .
	\end{align}}  Then  there exists $C>0$ depending only on $\rho,\,C_0,\,\nu$ such that
	\begin{align}\label{Ncond}
		\cN(f,p)\le C.
	\end{align}
\end{lemma}
\begin{proof}
	We have 
	\begin{align*}
		\cN(f,p) \le&  \sum_{m=0}^{\nu+2}	\norm{	\br{x}^{m-p+\rho}	f^{(m)}(x)} \int_\Rb \br{x}^{-1-\rho}dx
		\\\le& (\nu+3) C_0 \int_\Rb \br{x}^{-1-\rho}dx,
	\end{align*}
	and the integral converges for $\rho>0$. 
\end{proof}

%

Write $z=x+iy\in\Cb$   and $\di_{\bar z}=\di_x+i\di_y$. In what follows, 					as in \cite[eq.(B.5)]{HunSig3}, for $f\in C^{\nu+2}(\Rb)$, we take $\tilde f(z)$ to be an almost analytic extension of $f$ defined by 	
\begin{equation}\label{tfDef}
\widetilde f (z):=\eta\del{\frac{y}{\br{x}}}\sum_{k=0}^{\nu+1}f^{(k)}(x)\frac{(iy)^k}{k!},
\end{equation}							where $\eta\in C_c^\infty(\Rb)$ is a cutoff function   with  	$\eta(\mu)\equiv1$ for $\abs{\mu}\le1$,  $\eta(\mu)\equiv0$ for $\abs{\mu}\ge2$, and $\abs{\eta'(\mu)}\le1$ for all $\mu$.
This $\widetilde f(z)$ induces a measure on $\Cb$ as
\begin{align}\label{measDef}
d\widetilde f(z):=-\frac{1}{2\pi}\di_{\bar z}\widetilde f(z)dx\,dy.
\end{align}
In the remainder of this section, we derive integral estimate for various functions against the measure \eqref{measDef}.

The next result is obtained by adapting the argument  in \cite[Lem.~B.1]{HunSig3}:
\begin{lemma}[Remainder estimate]\label{lemRemEst}
Let $0\le p \le \nu$.
Let $f\in C^{\nu+2}(\Rb)$ satisfy \eqref{Ncond}.
Then {the extension $\widetilde f$ from \eqref{tfDef} satisfies} the following estimate for some 						  $C=C(f,\nu,p)>0:$
\begin{align}\label{fRemEst}
	\int \abs{d\widetilde{f}(z)} \abs{\Im(z)}^{-(p+1)} \le C.
\end{align}
\end{lemma}
\begin{proof}

Differentiating formula \eqref{tfDef}, we obtain the estimate
\begin{align}\label{dzfEst}
	\abs{\di_{\bar z} \widetilde f(z)}\le
	\eta\del{\frac{y}{\br{x}}}\frac{\abs{y}^{\nu+1}}{(\nu+1)!}\abs{f^{(\nu+2)}(x)}+\sum_{k=0}^{\nu+1}\rho\del{\frac{y}{\br x}}\frac{\abs{y}^k}{k!}\abs{\frac{1}{\br{x}}f^{(k)}(x)},
\end{align}
where \begin{align}\label{rhodef}
	\rho(\mu):=\abs{\eta'(\mu)}\br{\mu}
\end{align} is supported on $1<\abs{\mu}<2$. 

For each fixed $x$, we define
\begin{align}\label{Gxdef}
	G(x):=p.v.\int \abs{\di_{\bar z}f(z)}\abs{y}^{-(p+1)}\,dy
\end{align}
by integrating \eqref{dzfEst} against $\abs{y}^{-(p+1)}$.  Using that 	$\eta(y/\br x)\equiv 0$ for $\abs{y}>\br{x}$ and $\rho(y/\br{x})\equiv 0$ for $\abs{y}\le \br{x}$ or $\abs{y}\ge2\br{x}$, we find
\begin{align}
	G(x)\le&\int_{\abs{y}\le \br{x}} \frac{\abs{y}^{\nu-p}}{(\nu+1)!}\eta\del{\frac{y}{\br{x}}}\,dy\abs{f^{(\nu+2)}(x)}\label{549}
	\\&+\sum_{k=0}^{\nu+1}\int_{\br{x}<\abs{y}<2 \br{x}} \rho\del{\frac{y}{\br x}}\frac{\abs{y}^{k-p-1}}{k!}\,dy\abs{\frac{1}{\br{x}}f^{(k)}(x)}.\label{549'}
\end{align}
Since $0\le \eta(\mu)\le1$ and $\nu\ge p$, the integral in line \eqref{549} converges and can be bounded as
\begin{align}
	\int_{\abs{y}\le \br{x}} \frac{\abs{y}^{\nu-p}}{(p+1)!}\eta\del{\frac{y}{\br{x}}}\,dy\abs{f^{(p+2)}(x)}\le\frac{2\br{x}^{\nu-p+1}}{(p+1)!}\abs{f^{(p+2)}(x)}.\label{RemEst1}
\end{align}

To bound line \eqref{549'}, we use that $\rho(y/\br{x})< \sqrt{5}$ and $\abs{y}^{k-p-1}\le \br{x}^{k-p-1}$ for $\br{x}<\abs{y}<2\br{x}$,  $0\le k\le p+1$ (see \eqref{rhodef}). Thus  each integral in line \eqref{549'} can be bounded as
\begin{align}
	&\sum_{k=0}^{\nu+1}\int_{\br{x}<\abs{y}<2 \br{x}} \rho\del{\frac{y}{\br x}}\frac{\abs{y}^{k-p-1}}{k!}\,dy\abs{\frac{1}{\br{x}}f^{(k)}(x)}\notag\\\le&{\sum_{k=0}^{p+1}  \frac{4\sqrt{5} \br{x}^{k-p-1}}{k!}\abs{f^{(k)}(x)}
		+\sum_{k=p+1}^{\nu+1}  \frac{\sqrt{5}\cdot 2^{k-p+1}\br{x}^{k-p-1}}{k!}\abs{f^{(k)}(x)}}
	.\label{RemEst2}
\end{align}
Combining \eqref{RemEst1}--\eqref{RemEst2} in \eqref{549'}, we conclude that  
\begin{align}
	\abs{G(x)}\le C F(x),\quad F(x):= \sum_{m=0}^{\nu+2}  \br{x}^{m-p-1}\abs{f^{(m)}(x)}.
\end{align}

Let $G_\lambda(x):=\one_{[-\lambda,\lambda]} G(x)$ with $\lambda>0$. Then   $G_\lambda\in L^1$ and $\abs{G_\lambda(x)}\le CF(x)$ for any $\lambda$. By assumption \eqref{Ncond} and definition\eqref{pnNorm}, we have $\norm{F}_{L^1}=\cN(f,p)<\infty$ and so $F\in L^1$. Therefore, sending $\lambda\to\infty$ and using the dominated convergence theorem   yields $G\in L^1$ with \begin{align}\label{GFest}
	\norm{G}_{L^1}\le C\norm{F}_{L^1}.
\end{align} Recalling definition \eqref{Gxdef}, we find $(2\pi)^{-1}\norm{G}_{L^1}=$l.h.s. of \eqref{fRemEst}. 
Thus we conclude   \eqref{fRemEst} from \eqref{GFest}.
\end{proof}


\subsection{Commutator expansions}\label{4.sec:commut}

In this section, we take 
$\widetilde f(z)$,  $d\widetilde f(z)$ to be as in \eqref{tfDef}--\eqref{measDef}.

We frequently use the following result, taken from \cite[Lems. B.2]{HunSig3}:
\begin{lemma}
\label{lemHSj-rep}
Let $f\in C^{\nu+2}(\Rb)$ satisfy \eqref{Ncond} for some $p\ge0$. 
Then for any self-adjoint operator $A$ on $\hf$, 
\begin{align}
	\label{HSj-rep}
	\frac{1}{p!}f^{(p)}(A)=\int_\Cb d\widetilde f(z)(z-A)^{-(p+1)},
\end{align}
where the integral   converges absolutely in operator norm and is uniformly bounded  in $A$.  
\end{lemma}
{							\begin{remark}
	Condition \eqref{Ncond} ensures that $f^{(p)}$ is bounded independent of $A$ and  the remainder estimate in \lemref{lemRemEst} ensures the norm convergence of the r.h.s. of \eqref{HSj-rep}.
\end{remark}}

We call equation \eqref{HSj-rep} the \textit{Helffer-Sj\"ostrand (HS) representation}.  The HS representation \eqref{HSj-rep}, together with the remainder estimate \eqref{fRemEst}, implies the following commutator expansion:
\begin{lemma}\label{lemA.2} 
Let $n\ge1$.
Let $f\in C^{n+3}(\Rb)$ satisfy \eqref{Ncond} {with $p=1$}  . Let $A$ be an operator on $\hf$. Let $\phi$ be a densely defined self-adjoint operator on $\hf$. {Let $\A_s(f):=f(s^{-1}(\phi-\al))$ for some fixed $\al$ and all $s>0$.}

Suppose 
\begin{equation}\label{A.4}
	B_k:=		\ad{k}{\phi}{A}\in\cB(\hf)\quad (1\le k\le n+1).
\end{equation}
Then $[A, \A_s(f)]\in\cB(\hf)$, and we have the expansion
\begin{align}  [A, \A_s(f)]= &\sum_{k=1}^n(-1)^k\frac{s^{- k}} {k!}B_k\A_s(f^{(k)}) +(-1)^{n+1}s^{-(n+1)}\Rem_{\rm left}(s)\\
	=&
	\sum_{k=1}^n \frac{s^{- k}}{k!}\A_s(f^{(k)}) B_k + s^{-(n+1)}\Rem_{\rm right}(s),\label{commex} 
\end{align}
where 
the remainders are defined by these relations and given explicitly by \eqref{left-rem}--\eqref{right-rem}.

Moreover,  	 there exists $c>0$ depending only on $n$ and  $\cN(f,n+1)$, such that
\begin{align}
	\norm{\Rem_{\rm left}(s)}_{\rm op}+\norm{\Rem_{\rm right}(s)}_{\rm op}\le& c\norm{B_{n+1}},\label{4.A.21}
\end{align}

\end{lemma} 
\begin{remark}
Note that $f$ needs not to be bounded. By \eqref{fCond}, it suffices for $f$ to have strictly sublinear growth.
\end{remark}
\begin{proof}[Proof of \lemref{lemA.2}]
Within this proof we write $R=(z-\phi_{s,\al})^{-1}$ with  $\phi_{s,\al}=s^{-1}(\phi-\al)$.		

Since $R$ is bounded,  it follows that 
\begin{equation}\label{4.A.26}
	\big[ A , R \big ]= s^{-1}R\Ad_\phi(A)R
\end{equation}
holds in the sense of quadratic forms on $\mathcal{D}(A)$. {Since $\Ad_\phi(A)$ is bounded by assumption, the r.h.s. of \eqref{4.A.26} is bounded and so $[A,R]$ extends to an bounded operator on $\hf$.} Using \eqref{4.A.26}, we   proceed by commuting successively the commutators $B_k:=\ad{k}{\phi}{A}$ to left and right, respectively. Iteratively, we  obtain
{					\begin{align}
		&[A , R]\notag
		\\=& \sum_{k=1}^n (-1)^ks^{-k}B_k R^{k+1} +(-1)^{n+1}s^{-(n+1)}RB_{n+1}R^{n+1} \label{B25}
		\\=&\sum_{k=1}^n  s^{-k}R^{k+1}B_k + s^{-(n+1)}R^{n+1}B_{n+1}R,\label{B26}
	\end{align}					which hold on all of $\hf$} since  $B_k$'s are bounded operators by assumption \eqref{A.4}.

Let $\eta^\lambda\in C_c^\infty(\Rb)$, $\lambda>0$ be cutoff functions   with  	$\eta^\lambda(x)\equiv1$ for $\abs{x}\le\lambda$,  $\eta(x)\equiv0$ for $\abs{\mu}\ge\lambda+1$, and $\norm{\eta^\lambda}_{C^{n+3}}\le C$ for all $\lambda$.
Set $f^\lambda:= \eta^\lambda f$.
Since $f^\lambda\in C_c^{n+3}$, it satisfies \eqref{Ncond} for all $p\ge0$. (Note that $f$ itself, a priori, does not satisfy \eqref{Ncond} with $p=0$.) Thus the HS representation  \ref{HSj-rep} holds with $p=0$ and so 
\begin{align}\label{4.HSj-rep7}
	[A, \A_s(f^\lambda)]=\int d\widetilde {f^\lambda}(z) \big[ A , R \big ],
\end{align}
which holds a priori on $\cD(A)$. 
Plugging expansions \eqref{B25}--\eqref{B26} into \eqref{4.HSj-rep7} yields
\begin{align}
	&\quad[A, \A_s(f)^\lambda]\notag\\&= \sum_{k=1}^n(-1)^k\frac{s^{- k}}{k!} B_k\int d\widetilde {f^\lambda}(z) R^{k+1} +(-1)^{n+1}s^{-(n+1)}\Rem_{\rm left}^\lambda(s),\label{4.HSj-rep9}\\
	&=\sum_{k=1}^n \frac{s^{- k}}{k!}\int d\widetilde {f^\lambda}(z) R^{k+1} B_k+ s^{-(n+1)}	\Rem_{\rm right}^\lambda(s)\label{4.A.20},
\end{align}
where
\begin{align}
	\Rem_{\rm left}^\lambda(s)&=\int d\widetilde {f^\lambda}(z)RB_{n+1}R^{({n+1})},\label{left-rem}\\
	\Rem_{\rm right}^\lambda(s)&=\int d\widetilde {f^\lambda}(z)R^{({n+1})}B_{n+1}R\label{right-rem}. 
\end{align}    
Since the operator $B_{n+1}$ is bounded independent of $\lambda,\,z$, and $\norm{R}\le \abs{\Im(z)}^{-1}$, we have
\begin{align}\label{4.A.23}
	\notag  &\norm{\Rem_{\rm left}^\lambda(s)}_{\rm op}+\norm{\Rem_{\rm right}^\lambda(s)}_{\rm op}\\\le& 2 \| B_{n+1}\|  \int |d\widetilde {f^\lambda}(z)| \norm{R}_\op^{n+2}\notag\\
	\le &2 \| B_{n+1}\|  \int |d\widetilde {f^\lambda}(z)| |\Im(z)|^{-(n+2)}.
\end{align}
Similarly we could bound the sums in \eqref{4.HSj-rep9}--\eqref{4.A.20}. Thus we see $[A,\A_s(f^\l)]$ extends to a bounded operator on $\hf$ for each $\lambda$.

By \eqref{Norder} and the assumption $\cN(f,1)\le C$, $f$ satisfies condition \eqref{Ncond} with $p=1,\ldots, n+1$. Hence,  sending $\lambda\to \infty$ in \eqref{4.HSj-rep9}--\eqref{right-rem} and using \eqref{HSj-rep} for $p=1,\ldots, n$ and remainder estimate \eqref{fRemEst} for $p=n+1$, we conclude that $[A,\A_s(f)]\in \cB(\hf)$ and  expansions 
\eqref{commex} and  estimate   \eqref{4.A.21} hold.\end{proof}

The following lemma is a direct consequence of the estimates proved in \lemref{lemA.2}:
\begin{lemma}\label{lem5.5}
	Let $G(t)$ be defined as in \eqref{commVphi}. Then for any $s>0$, $\chi\in\cE$, the commutator $[\A_s(\chi),V(t)]$ extends to bounded operators for all times and satisfies, for some absolute constant $C>0$ ,
	\begin{align}
		\label{AVcommEst}
 	\norm{[V(t),\A_s(\chi)]}\le Cs^{-1}G(t) \qquad (t\in\Rb).
	\end{align}
\end{lemma}
\begin{proof}
	
	By relations \eqref{4.A.26} and \eqref{4.HSj-rep7}, we have for all $t\in\Rb$ that
	\begin{align}
		\label{}
		[V(t),\A_s(\chi)]=s^{-1}\int d\widetilde {\chi^\lambda}(z) R\big[ V(t) , \phi \big ]R.
	\end{align}
	This, together with the definition $G(t)\equiv \norm{[\phi,V(t)]}$, yields
	\begin{align}
		\label{533}
	 	\norm{[V(t),\A_s(\chi)]}\le   s^{-1} G(t)\int |d\widetilde {\chi^\lambda}(z)| |\Im(z)|^{-2} .
	\end{align}
	Since the integral in \eqref{533} is uniformly bounded by some absolute constant for all $\l$,   sending $\l\to\infty$ yields the desired estimate \eqref{AVcommEst}. 
\end{proof}

	\section*{Acknowledgments}

The Author is supported by National Key R \& D Program of China Grant 2022YFA100740, China Postdoctoral Science Foundation Grant 2024T170453,  and the Shuimu Scholar program of Tsinghua University.
He was also supported by DNRF Grant CPH-GEOTOP-DNRF151, DAHES Fellowship Grant 2076-00006B, DFF Grant 7027-00110B, and the Carlsberg Foundation Grant CF21-0680 during the completion of this paper. 
He thanks G.~Grubb for introduction to the subject matter in \secref{chapMVENonLoc}, R.~Frank for helpful discussion and pointing out reference \cite{MR0177312}, and C.~Jiang for useful comments on earlier draft of the paper. He thanks S.~Breteaux, J.~Faupin, M.~Lemm, C.~Rubiliani, M.~Sigal,  and D.~Ouyang for fruitful collaborations, and especially  M.~Lemm for pointing out \eqref{tailEst} as well as other helpful comments.

Parts of this work were done while the Author was visiting MIT. Earlier version of parts of this work has appeared as a chapter in the Author's PhD thesis at the University of Copenhagen.

		\section*{Declarations}
\begin{itemize}
	\item Conflict of interest: The Author has no conflicts of interest to declare that are relevant to the content of this article.
	\item Data availability: Data sharing is not applicable to this article as no datasets were generated or analysed during the current study.
\end{itemize}

\appendix

\section{Commutator estimates}\label{sec:A}

In this appendix, we prove that condition \eqref{k-cond1} implies uniform estimates on multiple commutators with (multiplication operator by) Lipschitz functions. In particular, this implies that \eqref{k-cond0} holds with $H_0$ from \eqref{1.2} and $\phi=d_X$.

\begin{lemma}
	\label{lemA.1}
	Let $n\ge1$. Suppose $A$ is an operator acting on  $L^2(\Rb^d)$ as
	\begin{equation}\label{A.0}
		A[u](x)=\int_{\Rb^d} (V(x)u(x)-u(y))K(x,y)\,dy
	\end{equation}
	for $V\in L^\infty(\Rb^d)$ and integral kernel $K(x,y)$ satisfying
	\begin{align}
		\label{A.1}
		M:=&\sup_{1\le p\le n+1}\del{\sup _{x\in\Rb^d} \int_{y\in\Rb^d}  \abs{K(x,y)}\abs{x-y}^{p}}\del{\sup _{y\in\Rb^d} \int_{x\in\Rb^d} \abs{K(x,y)}\abs{x-y}^{p}}<\infty.
	\end{align}
	{Then for every Lipschtiz function $f$} on $ \Rb^d$ such that for some $L>0$,

	\begin{equation}\label{phicond}
		\abs{f(x)-f(x)}\le L \abs{x-y}\quad (x,y\in\Rb^d),
	\end{equation}
	there holds
	\begin{equation}\label{CommEst}
		\norm{\ad{k}{f}{A}}\le L^k M\quad (1\le k\le n+1).
	\end{equation}
\end{lemma}
\begin{proof}
	We first prove that for each fixed $f:\Rb^d\to \Cb$ and all $1\le k\le n+1$, we have 
	\begin{equation}\label{A.2}
		\ad{k}{f}{A}[u]=-\int(f(y)-f(x))^kK(x,y)u(y)\,dy.
	\end{equation}
	We prove this by a simple induction.  Clearly, the $V$ term  in \eqref{A.0} does not contribute to the commutators $	\ad{k}{f}{A}$, since $[V,f]\equiv 0$. Hence below we take $V\equiv 0$  in \eqref{A.0}.

	For the base case $k=1$, we compute, for fixed $f$ and every $u$, 
	\begin{align*}
		A[fu](x)=& -\int K(x,y)f(y)u(y)\,dy,\\
		f(x)		A[u](x)=& -\int f(x)K(x,y)u(y)\,dy.
	\end{align*}
	Taking the difference yields \eqref{A.2} with $k=1$. Now assume \eqref{A.2} holds for $k$. Then we have
	\begin{align*}
		\ad{k}{f}{A}[fu](x)=-& \int (f(y)-f(x))^k K(x,y)f(y)u(y)\,dy,\\
		f(x)		\ad{k}{f}{A}[u]= -&\int f(x)(f(y)-f(x))^kK(x,y)u(y)\,dy.
	\end{align*}
	Since $\ad{k+1}{f}{A}=[\ad{k}{f}{A},f] $, taking the difference of the last two expressions yields \eqref{A.2} for $k+1$. This completes the induction.

	Formula \eqref{A.2}, together with the Schur test for integral operators, implies
	\begin{align}\label{A.3}
		\norm{\ad{k}{f}{A}}^2\le& \del{\sup _{x\in\Rb^d} \int_{y\in\Rb^d} \abs{K(x,y)}\abs{f(x)-f(y)}^k}   \del{\sup _{y\in\Rb^d} \int_{x\in\Rb^d} \abs{K(x,y)}\abs{f(x)-f(y)}^k}.
	\end{align}
	Now we compute, using assumptions \eqref{A.1} and \eqref{phicond}, that
	\begin{align*}
		&\sup _{x\in\Rb^d} \int_{y\in\Rb^d} \abs{K(x,y)}\abs{f(x)-f(y)}^k \le L^k\sup _{x\in\Rb^d} \int_{y\in\Rb^d} \abs{K(x,y)}\abs{x-y}^k\le L^kM .
	\end{align*}
	This bounds the first  term in the r.h.s. of \eqref{A.3}. Similarly we can derive the same bound for the second term in the r.h.s. of \eqref{A.3}. Plugging the results back to \eqref{A.3} yields estimate \eqref{CommEst}.
\end{proof}

\bibliography{MVBNonLocBibfile}

	\end{document}